\documentclass [12pt,twoside,a4paper]{article}
\usepackage{amsmath}
\usepackage{amsfonts}
\usepackage{amsthm}
\usepackage{amstext}
\usepackage{amssymb}
\usepackage{mathrsfs}
\usepackage{amscd}
\usepackage{xypic}
\usepackage{cite}
\usepackage[numbers,sort&compress]{natbib}
\usepackage{epsf}              
\usepackage{fancybox}          
\usepackage{color}             
\usepackage{fancyhdr}
\usepackage[hang,footnotesize]{caption2}  
\usepackage{enumerate} 
\usepackage{float}
\usepackage{makecell}
\usepackage{subfigure}
\usepackage{graphicx}
\usepackage{epstopdf}
\usepackage{epsfig}
\usepackage[numbers]{natbib}
\usepackage{ragged2e}
\usepackage{graphicx,epsfig,epstopdf,subfigure}
\numberwithin{equation}{section}
\newfont{\aaa}{cmb10 at 19pt}
\newfont{\bbb}{cmb10 at 14pt}
\newtheorem{theorem}{Theorem}[section]
\newtheorem{corollary}[theorem]{Corollary}
\newtheorem{conjecture}[theorem]{Conjecture}
\newtheorem{lemma}[theorem]{Lemma}
\newtheorem{fact}[theorem]{Fact}
\newtheorem{claim}[theorem]{Claim}

\newtheorem{proposition}[theorem]{Proposition}
\newtheorem{problem}[theorem]{Problem}

\makeatletter

\newcommand{\Rmnum}[1]{\expandafter\@slowromancap\romannumeral #1@}
\makeatother

\newcommand{\bal}{\begin{align}}
\newcommand{\eal}{\end{align}}
\newcommand{\beq}{\begin{equation}}
\newcommand{\eeq}{\end{equation}}
\newcommand{\bey}{\begin{eqnarray}}
\newcommand{\eey}{\end{eqnarray}}
\newcommand{\beyy}{\begin{eqnarray*}}
\newcommand{\eeyy}{\end{eqnarray*}}

\bfseries\normalfont

\begin{document}

\title{Spanning $H$-subdivisions and perfect $H$-subdivision tilings in dense digraphs 
}

\author{Yangyang Cheng\textsuperscript{12}, Zhilan Wang\textsuperscript{1}, Jin Yan\textsuperscript{1}\thanks{Corresponding author. E-mail address: yanj@sdu.edu.cn.}
\unskip\\[2mm]
}
\date{}
\maketitle
\begin{abstract}
\noindent Given a (di)graph $H$, we say that a (di)graph $H^\prime$ is an $H$-subdivision if $H^\prime$ is obtained from $H$ by replacing one or more edges with internally vertex-disjoint path(s). Pavez-Sign\'{e} conjectured that for every $\varepsilon>0$, there exists a constant $C_0>0$ such that for every graph $H$ with $h$ edges and no isolated vertices, if $G$ is a graph on $n\geq C_0h$ vertices and minimum degree $\delta(G)\geq(1+\varepsilon)\frac{n}{2}$, then $G$ contains a spanning $H$-subdivision. This conjecture was later resolved by Lee [European J. Combin. \textbf{124} (2025), 104059]. In this paper, we strengthen Lee's result. Specifically, we prove that for any digraph $D$ on $n\geq C_0h$ vertices, if the minimum semi-degree of $D$ is at least $\frac{n+h}{2}-1$, then $D$ contains a spanning $H$-subdivision. The lower bound on the minimum semi-degree is best possible.

Furthermore, we show that there exist constants $C>0$ and $\alpha, \beta\in(0, 1)$ such that for any integer partition $n=n_1+\cdots+n_m\geq Cm$ with $n_i\geq|V(H)|+3h$ for each $i$, and $\sum_{n_i<\alpha n}n_i\leq\beta n$, if a digraph of order $n\geq Cm$ has minimum semi-degree at least $\frac{n+m+h}{2}-1$, then it contains $m$ vertex-disjoint $H$-subdivisions whose orders are $n_1, \ldots, n_m$, respectively. The bound $\frac{n+m+h}{2}-1$ is also optimal. This work partly answers a conjecture of Lee [Combin. Probab. Comput. \textbf{34} (2025), 421--444] and generalizes a recent result from the same paper.
\end{abstract}

\section{Introduction}
Given a (di)graph $H$, we say that a (di)graph $H^\prime$ is an \emph{$H$-subdivision} if $H^\prime$ is obtained from $H$ by replacing one or more edges with internally vertex-disjoint path(s). Also, we call an $H$-subdivision \emph{spanning} if it covers all vertices of the host (di)graph.
Because subdivisions maintain the topological structure of the original (di)graph, various problems related to subdivisions were proposed and extensively studied. A particular focus has been on identifying a sufficient condition that ensures the host graph contains a subdivision of $H$, making this an important topic in extremal graph theory.

\smallskip

Recently, Pavez-Sign\'{e} \cite{Pavez} showed that for every $\varepsilon>0$, there exists a constant $C_0$ such that for all $C\geq C_0$ and $k\geq2$ the following holds. Let $\log k\leq d\leq k$ and let $G$ be a graph on $n=Cdk$ vertices and minimum degree at least $(1+\varepsilon)\frac{n}{2}$. Then $G$ contains a spanning subdivision of every $k$-vertex $d$-regular graph. Here, a graph is \emph{$d$-regular} if every vertex in it has the degree $d$. In the same paper, Pavez-Sign\'{e} also proposed the following conjecture.
\begin{conjecture}$($[\citenum{Pavez}, Conjecture~3.1]$)$\label{wenti}
For every $\varepsilon>0$, there exists a positive constant $C_0$ such that for all $C\geq C_0$ and any integer $h$ the following holds. Let $G$ be a graph on $n=Ch$ vertices and minimum degree at least $(1+\varepsilon)\frac{n}{2}$. Then, $G$ contains a spanning $H$-subdivision of every graph $H$ with $h$ edges and no isolated vertices.
\end{conjecture}
We define the \emph{minimum semi-degree} of a digraph $D$ as $\delta^0(D)=\min\{\delta^+(D), \delta^-(D)\}$ and the \emph{minimum degree} $\delta(D)=\min_{x\in V}\{d(x): d(x)=d^+(x)+d^-(x)\}$. Lee \cite{Lee} solved  Conjecture \ref{wenti} in a stronger form by considering digraphs.
\begin{theorem}$($[\citenum{Lee}, Theorem~1.2]$)$\label{lee}
Let $\varepsilon>0$ be a positive real number. Then there exists a constant $C_0>0$ such that for all $C\geq C_0$ the following holds. Let $D$ be a digraph on $n\geq Ch$ vertices and $\delta^0(D)\geq(\frac{1}{2}+\varepsilon)n$. Then $D$ contains a spanning subdivision of any digraph $H$ with $h$ arcs and no isolated vertices.
\end{theorem}
By replacing each edge in the graphs $G$ and $H$ from Conjecture \ref{wenti} with a pair of oppositely directed edges, one can directly deduce that Conjecture \ref{wenti} follows from Theorem \ref{lee}. In this paper, we improve the result above of Lee, by considering the $H$-subdivision problem in general digraphs under a sharp minimum semi-degree condition.
\begin{theorem}\label{song1}
There exists a constant $C_0>0$ such that for every integer $C\geq C_0$ and any digraph $H$ with $h$ arcs and minimum degree $\delta(H)\geq1$, the following holds. If $D$ is a digraph on $n\geq Ch$ vertices with  $\delta^0(D)\geq\frac{n+h}{2}-1$, then $D$ contains a spanning $H$-subdivision.
\end{theorem}
Clearly, if $H$ is a $2$-cycle, then a spanning $H$-subdivision in $D$ is exactly a Hamilton cycle in $D$. Since the tight lower bound of $\delta^0(D)$ for $D$ to contain a Hamilton cycle is $n/2$, the semi-degree bound in Theorem \ref{song1} is tight.

\smallskip

Theorem \ref{song1} also allows us to make progress on the following problem posed by Babu and Diwan \cite{Diwan}, in a stronger form.
\begin{problem}\cite{Diwan}
For every fixed graph $H$ with no isolated vertices, does there exist an integer $f(H)$ such that every graph of order $n\geq f(H)$ and minimum degree at least $(n+1)/2$ contains a spanning $H$-subdivision?
\end{problem}

The spanning subgraph structure is a special case of perfect subgraph tiling. K\"{u}hn and Osthus \cite{Kuhn} gave an almost exact value of minimum degree to guarantee that a graph contains a perfect $H$-tiling for every graph $H$ up to an additive constant depending only on $H$. There are several results related to the result of K\"{u}hn and Osthus, see references \cite{Balogh,Freschi,Han,Hurley,Hyde,Kuhn1,Lo}.

\smallskip

In particular, given a digraph $H$, a \emph{perfect $H$-subdivision tiling} in a digraph $D$ is a collection of vertex-disjoint subdivisions of $H$ covering all vertices of $D$.
It is natural to determine what the minimum degree threshold is to ensure the existence of a perfect $H$-subdivision tiling in a digraph. We use standard hierarchy notation, that is, we write $0<\alpha\ll\gamma$ to indicate that $\alpha$ is chosen to be sufficiently small
relative to $\gamma$ such that all calculations needed in our proof are valid. In this paper, we also give the the following result.
\begin{theorem}\label{song2}
Let $H$ be a digraph with $h$ arcs and $\delta(H)\geq1$. For any constants $\alpha, \beta$ with $0<\alpha, \beta\ll1$, there exists a constant $C_0=C_0(H, \alpha, \beta)$ such that the following holds for every integer $C\geq C_0$ satisfying $0<1/C\ll \min\{\alpha, \beta\}$. Consider any integer partition $n=n_1+\cdots+n_m$ such that $n_i\geq3h+|V(H)|$ for each $i$, and $\sum_{n_i<\alpha n}n_i\leq\beta n$. If $D$ is a digraph on $n\geq Cm$ vertices with $\delta^0(D)\geq\frac{n+m+h}{2}-1$, then $D$ contains a perfect $H$-subdivision tiling, in which, the orders of these $H$-subdivisions are $n_1, n_2, \ldots, n_m$, respectively. 
\end{theorem}
The bound on the minimum semi-degree is tight. We construct a counterexample following the idea from  \cite{Molla}. Consider the case where $H$ is $2$-cycle. Here, the bound $\frac{n+m+h}{2}-1$ simplifies to $\frac{n+m}{2}$, and a perfect tiling of $H$-subdivisions is equivalent to a perfect tiling of cycles. Let $D$ be a complete $3$-partite digraph with vertex classes $X, Y, Z$, where $|X|=m-1$, $|Y|=(n-m+2)/2$ and $|Z|=(n-m)/2$. Then $|V(D)|=n$. The semi-degrees are as follows: for any $x\in X$, $d^+(x)=d^-(x)=n-m+1$; for any $y\in Y$, $d^+(y)=d^-(y)=\frac{n+m}{2}-1$; and for any $z\in Z$, $d^+(z)=d^-(z)=\frac{n+m}{2}$. Given that $n\geq Cm$ and $1/C\ll1$, we have $n-m+1>\frac{n+m}{2}-1$. Consequently, the minimum semi-degree is $\delta^0(D)=\frac{n+m}{2}-1$. We now show that $D$ contains no perfect tiling of $m$ disjoint odd cycles. The key observation is that in this digraph, every odd cycle must contain at least one vertex from $X$. Since $|X|=m-1<m$,  it is impossible to pack $m$ such cycles. This completes the counterexample, establishing that the bound is best possible.

\smallskip

In the same way, replacing edges of the graphs $G$ and $H$ with two arcs in both directions, we obtain the following corollary from the theorem above. This corollary extends a result of Lee \cite{Lee1} that recently gave an answer to the question on perfect $H$-subdivision tilings for sufficiently large graphs. Furthermore, it provides a partial answer to the conjecture raised in his same paper \cite[Conjecture~6.1]{Lee1}, which predicts the exact minimum semi-degree that guarantees the existence of a perfect $H$-subdivision tiling.
\begin{corollary}\label{el}
There exists a constant $C_0>0$ such that for every integer $C\geq C_0$ and two constants $\alpha, \beta$ with $0\ll1/C\ll\alpha, \beta\ll1$, the following holds. Suppose $H$ is a graph with $h$ edges and $\delta(H)\geq1$. For any integer partition $n=n_1+\cdots+n_m$ such that $n_i\geq3h+|V(H)|$ for each $i$, and $\sum_{n_i<\alpha n}n_i\leq\beta n$, if $G$ is a graph on $n\geq Cm$ vertices with  $\delta(G)\geq\frac{n+m+h}{2}-1$, then $G$ contains a perfect $H$-subdivision tiling of order $n_1, \ldots, n_m$, respectively.
\end{corollary}
Theorem \ref{song2} can also derive many disjoint subdigraph structures. For example, if $H$ is a cycle of length $2$, then Theorem \ref{song2} implies that there exist constants $C_0>0$ and $\alpha, \beta$ with $0\ll1/C_0\ll\alpha, \beta\ll1$ such that the following holds. For any integer partition $n=n_1+\cdots+n_m\geq C_0m$ such that $n_i\geq8$ for each $i$, and $\sum_{n_i<\alpha n}n_i\leq\beta n$, if $D$ is an $n$-vertex digraph with $\delta^0(D)\geq\frac{n+m}{2}$, then it contains $m$ disjoint cycles of length $n_1, \ldots, n_m$, respectively. For other cycle results in digraphs, we refer to the references \cite{Czygrinow,keevash,Kelly,Keevash1,Li,Molla,Wang2,Wang1}.

\smallskip

Also, it is worth noting that in the proofs of Theorems \ref{song1} and \ref{song2}, we need to utilize the absorption method that was introduced by R\"{o}dl, Ruci\'{n}ski and Szemer\'{e}di \cite{Rodl} as well as the stability method. Also, this work is a follow-up work to \cite{Wang}, and the proof method in this paper closely follows the strategy in \cite{Wang}.

\medskip

\noindent \textbf{Organization.} The rest of the paper is organised as follows. In Section $2$, we begin by presenting relevant definitions and notations, followed by an outline of the proofs of Theorems \ref{song1} and \ref{song2}. Section $3$ presents the proofs of Theorems \ref{song1} and \ref{song2} for non-extremal case. In Subsection $3.1$, we prove the absorbing and path-covering lemmas. In Subsection $3.2$, we complete the proofs Theorems \ref{song1} and \ref{song2} for non-extremal case. In Section $4$, we first identify three extremal cases that $D$ can belong to when it satisfies the extremal condition, and then we will prove that Theorems \ref{song1} and \ref{song2} hold in each of these cases. Finally, Section $5$ contains some concluding remarks to wrap up the paper.
\section{Preparations for Theorems \ref{song1} and \ref{song2}}
\subsection{Definitions and notations}
For notations not defined in this paper, we refer the reader to \cite{Bang-Jensen3}.
Let $D=(V, A)$ be a digraph. The \emph{out-neighbourhood} (resp., \emph{in-neighbourhood}) of a vertex $v$ in $D$ is defined as $N^{+}(v)=\{u: vu\in A\}$ (resp., $N^{-}(v)=\{w: wv\in A\}$). The \emph{out-degree} (resp., \emph{in-degree}) of $v$ in $D$, which is denoted by $d^+(v)$ (resp., $d^-(v)$), is the cardinality of $N^{+}(v)$ (resp., $N^{-}(v)$), that is, $d^{+}(v)=|N^{+}(v)|$ (resp., $d^{-}(v)=|N^{-}(v)|$). The \emph{minimum out-degree} $\delta^+(D)=\min\{d^{+}(v): v\in V\}$ and the \emph{minimum in-degree} $\delta^-(D)=\min\{d^{-}(v): v\in V\}$. For any $X\subseteq V$ and $\sigma\in\{-, +\}$, we define $N^\sigma(u, X)=N^\sigma(u)\cap X$ and $d^\sigma_X(u)=|N^\sigma(u, X)|$ for any vertex $u$ in $V$. The cardinality of  $X$ is denoted by $|X|$, and we say $X$ is an $i$-set if $|X|=i$. The subdigraph of $D$ induced by $X$ is denoted as $D[X]$. For another vertex set $Y$ that is not necessarily disjoint from $X$, we use $e^+(X, Y)$ to represent the number of arcs from $X$ to $Y$.

\smallskip

Given digraphs $D$ and $H$, let $\mathcal{P}(D)$ be the family of paths in $D$. An \emph{$H$-subdivision} in $D$ is a pair of mappings $f: V(H)\rightarrow
V(D)$ and $g: A(H)\rightarrow \mathcal{P}(D)$ such that $(a)$ $f(u)\neq f(v)$ for $u, v\in V(H)$ with
$u\neq v$, and $(b)$ for every edge $uv\in A(H)$, $g(uv)$ is a path from $f(u)$ to $f(v)$, and different edges are mapped into internally vertex-disjoint paths in $D$.

\smallskip

We define the number of arcs of a path as its \emph{length} and a \emph{$k$-path} refers to a path of order $k$. We often represent the $k$-path $P$ as $v_1\cdots v_k$ when $V(P)=\{v_1, \ldots, v_k\}$ and call $v_1$ and $v_k$ the \emph{initial} and the \emph{terminal} of $P$, respectively. Furthermore, for two disjoint vertex sets $X$ and $Y$ in $V$, if the initial and the terminal of $P$ belongs to $X$ and $Y$, respectively, then we say that $P$ is an $(X, Y)$-path. In particular, we write an $X$-path instead of $(X, X)$-path if $X=Y$. Additionally, we say an $(X, Y)$-path $P$ is \emph{minimal} if there is no $(X, Y)$-path $P^\prime$ with $|V(P^\prime)|<|V(P)|$. All paths in digraphs refer to directed paths, and we use the term \emph{disjoint} instead of vertex-disjoint for simplicity. Given a family of graphs $\mathcal{F}$, denote by $|\mathcal{F}|$ the number of graphs in $\mathcal{F}$ and we write $V(\mathcal{F})=\bigcup_{T\in \mathcal{F}}V(T)$.

\smallskip

For a positive integer $t$, we simply write $\{1, \ldots , t\}$ as $[t]$. Throughout this paper, we will omit floor and ceiling signs when they are not essential. For the real numbers $a$ and $b$, we use $a\pm b$ to represent an unspecified real number in the interval $[a-b, a+b]$.

\smallskip

To summarize this subsection, we provide the following extremal condition for a constant $\varepsilon^\prime$, where $0<\varepsilon^\prime\ll1$. In particular, we say the digraph $D$ is \emph{stable with parameter $\varepsilon^\prime$} (For simplicity, \emph{$\varepsilon^\prime$-stable}) if $D$ does not satisfy the following extremal condition (\textbf{EC}).

\smallskip

\noindent \textbf{Extremal Condition with parameter $\varepsilon^\prime$ (EC($\varepsilon^\prime$)):} Let $D$ be a digraph of order $n$. There exist two (not necessarily disjoint) vertex sets $U_1$ and $U_2$ in $D$ with $|U_i|\geq(1/2-\varepsilon^\prime)n$ for every $i\in[2]$ such that $e^+(U_1, U_2)\leq(\varepsilon^\prime n)^2$.

\medskip

Based on the extremal condition above, we now introduce a key combinatorial lemma concerning the internal structure of digraphs, which will serve as an essential technical tool in the analysis of host digraphs satisfying this condition.

\smallskip

For a digraph $H$, a pair $(U_1, U_2)$ is called a \emph{bipartition} of $H$, written as $H=(U_1, U_2)$, if $U_1\cap U_2=\emptyset$ and $U_1\cup U_2=V(H)$, and $e(H)=e(U_1)\cup e(U_2)\cup e(U_1, U_2)$. Here, $e(X)$ denotes the set of arcs in the induced subdigraph $H[X]$ with $X\in\{U_1, U_2\}$, and $e(U_1, U_2)$ represents the arcs in the bipartite digraph $(U_1, U_2)$. Furthermore, we also define the quantity $m_H=\min_{H=(U_1, U_2)}\big||U_1|+e(U_2)-(|U_2|+e(U_1))\big|$.
\begin{lemma}\label{lem:bipartition}
Let $H$ be a digraph with $\delta^(H)\geq1$, and let $h:=e(H)\geq 2$. Then there exists a bipartition $(U_1, U_2)$ of $V(H)$ such that $\bigl||U_1|+e(U_2)-(|U_2|+e(U_1))\bigr|\leq h/2-1.$
\end{lemma}
\begin{proof}
For a bipartition $(U_1, U_2)$ of $V(H)$, define $f(U_1, U_2):=|U_1|+e(U_2)-(|U_2|+e(U_1))$. Choose a bipartition $(U_1, U_2)$ such that $|f(U_1, U_2)|$ is minimal. Suppose for a contradiction that $|f(U_1, U_2)|>h/2-1$. Since $|f(U_1, U_2)|$ is an integer, this implies $|f(U_1, U_2)|\geq \left\lfloor h/2\right\rfloor\geq 1$. By swapping $U_1$ and $U_2$ if necessary, we may assume that $t:=f(U_1, U_2)>0$.
\begin{claim}\label{h1}
Every vertex in $U_1$ has total degree at least $2$.
\end{claim}
\begin{proof}
Let $u\in U_1$. Moving $u$ from $U_1$ to $U_2$ yields a new bipartition $(U_1^\prime, U_2^\prime)$. The change in $f$ is $f(U_1^\prime, U_2^\prime)=f(U_1, U_2)+ d(u)-2=t+d(u)-2$. By the minimality of $|f(U_1, U_2)|$, we have  $|t+d(u)-2|\geq t$. Hence, either $d(u)-2\geq 0$ or $d(u)-2\leq -2t$. If the latter holds, then $d(u)\leq 2-2t\leq 0$ (since $t\geq 1$), contradicting the assumption that $\delta(H)\geq1$. Therefore, $d(u)\geq 2$ for all $u\in U_1$.
\end{proof}
\begin{claim}\label{h2}
 Every vertex in $U_2$ has total degree at most $2$.
\end{claim}
\begin{proof}
Let $v\in U_2$. Moving $v$ from $U_2$ to $U_1$ yielding a bipartition $(U_1^{\prime\prime}, U_2^{\prime\prime})$ with $f(U_1^{\prime\prime}, U_2^{\prime\prime})=t+2-d(v)$. Again, the minimality implies $|t+2-d(v)|\geq t$. Thus, either $2-d(v)\geq 0$ or $2-d(v)\leq-2t$. If the latter holds, then $d(v)\geq2+2t$. Recalling that $t\geq \left\lfloor h/2\right\rfloor$, we get $d(v)\geq2t+2\geq2\left\lfloor h/2\right\rfloor+2>h$, which contradicts the obvious fact that $d(v)\leq e(H)=h$. Hence, $d(v)\leq2$ for all $v\in U_2$.
\end{proof}

Since each arc in $U_1$ contributes $2$ to $\sum_{u\in U_1} d(u)$, each arc in $U_2$ contributes $2$ to $\sum_{v\in U_2} d(v)$, and each crossing arc contributes $1$ to each sum, we have
\begin{align*}
\sum_{u\in U_1} d(u)=2e(U_1)+e(U_1, U_2), \mbox{and}\ \sum_{v\in U_2} d(v)=2e(U_2)+e(U_1, U_2).
\end{align*}
Combining these with the degree bounds from Claims \ref{h1} and \ref{h2}, we obtain that
\begin{align}
&2|U_1|\leq\sum_{u\in U_1} d(u)=2e(U_1)+e(U_1, U_2),\ \mbox{and}\\
&\sum_{u\in U_2} d(v)=2e(U_2)+e(U_1, U_2)\leq 2|U_2|.
\end{align}
Recall that $h=e(H)=e(U_1)+e(U_2)+e(U_1, U_2)$ and $t=f(U_1, U_2)=|U_1|+e(U_2)-(|U_2|+e(U_1))$. We now derive bounds on the number of vertices $|V(H)|=|U_1|+|U_2|$. From $(4.5)$, we have $e(U_1,  U_2)\geq2|U_1|-2e(U_1)$. Substituting into the expression for $h$ gives $h\geq e(U_1)+e(U_2)+2|U_1|-2e(U_1)=e(U_2)-e(U_1)+2|U_1|$. Using the definition of $t$, we have $e(U_2)-e(U_1)=t-|U_1|+|U_2|$. Hence, $h\geq t-|U_1|+|U_2|+2|U_1|=t+|U_1|+|U_2|$. So, $|V(H)|\leq h-t$.

Similarly, for $(4.6)$ we obtain $e(U_1, U_2)\leq2|U_2|-2e(U_2)$. Substituting into $h$ yields $h\leq e(U_1)+e(U_2)+2|U_2|-2e(U_2)=e(U_1)-e(U_2)+2|U_2|$. Since $e(U_1)-e(U_2)=|U_1|-|U_2|-t$, we get $h\leq|U_1|-|U_2|-t+2|U_2|=|U_1|+|U_2|-t$. So $|V(H)|\geq h+t$. Hence, we obtain $h+t\leq |V(H)|\leq h-t$, which implies $t\leq0$. However, by our assumption for contradiction, we have $t>0$. Thus we have a contradiction. Therefore, the assumption is false, and we must have $t\leq \frac{h}{2}-1$. This completes the proof.
\end{proof}
\noindent\textbf{Observation.} The upper bound $h/2-1$ in Lemma \ref{lem:bipartition} is tight. This can be seen by considering the case where $H$ is a $2$-cycle. Then $m_H=0$ and $h/2-1=0$, which shows that the bound cannot be reduced further.
\subsection{Proof structure for the main theorems}
Let $H$ be a digraph with $h$ arcs and $\delta(H)\geq1$, and let $D$ be a digraph as described in Theorems \ref{song1} and \ref{song2}. Since the proofs of Theorem \ref{song1} and Theorem \ref{song2} follow a similar line of reasoning, we will only present the detailed proof strategy for Theorem \ref{song1}. The proof of Theorem \ref{song1} utilizes the stability method, which is divided into two main cases:

The \textbf{extremal case}, when the digraph $D$ is not $\varepsilon^\prime$-stable;

The \textbf{non-extremal case}, when $D$ is $\varepsilon^\prime$-stable.

\smallskip

$\bullet$ \textbf{Non-extremal case}

For the non-extremal case, we divide the proof into the following two steps:


\textbf{Step 1. Absorbing Lemma} (Construction of an $H$-subdivision subdigraph). By the extremal condition with parameter $\varepsilon$ and the probabilistic method, we construct an \emph{absorbing} subdigraph $H^\prime$ that is a subdivision of $H$ and possesses the remarkable property that for every vertex pair $(u, v)$ in $D-H^\prime$, every ``long'' subdivided path inside $H^\prime$ contains at least one absorber for $(u, v)$.

\textbf{Step 2. Path-Covering Lemma} (Lemma \ref{Path-Cover Lemma} in Subsection $3.1$). This lemma allows us to cover all vertices of $D-H^\prime$ with a bounded number of disjoint paths of arbitrary prescribed lengths.

Finally, by exploiting the absorbing property of $H^\prime$, we will absorb these disjoint paths of suitable lengths into $H^\prime$ to obtain the desired spanning $H$-subdivision. This completes the proof for the non-extremal case.

\smallskip

The proof of Theorem \ref{song2} differs from that of Theorem \ref{song1} only in Step $1$. Suppose that $n_1\geq\cdots\geq n_t\geq\alpha n>n_{t+1}\geq\cdots\geq n_m$. Then we will construct $t$ disjoint \emph{absorbing} $H$-subdivisions $H_1^\prime, \ldots, H_t^\prime$ of orders are at most $n_1, \ldots, n_t$, respectively, and subsequently construct $(m-t)$ disjoint $H$-subdivisions, whose orders are $n_{t+1}, \ldots, n_m$, respectively.

\smallskip

$\bullet$ \textbf{Extremal case}

\textbf{Step 3.} For the extremal case, we show via the conventional structural method that $D$ satisfies one of three extremal structures: \textbf{EC1($\varepsilon$)}, \textbf{EC2($\varepsilon$)}, or \textbf{EC3($\varepsilon, \varepsilon_1$)}. Their descriptions are as follows.

In EC1($\varepsilon$), with $O(\varepsilon n)$ exceptional vertices, the remaining vertices can be partitioned into two almost equal subsets $W_1$, $W_2$ such that each $D[W_i]$ is almost complete. In  EC2($\varepsilon$), analogously, after removing few exceptional vertices, the vertex set partitions into two almost equal subsets $W_1$, $W_2$ that form an almost complete bipartite digraph. In EC3($\varepsilon$, $\varepsilon_1$), again with few exceptional vertices, the remaining vertices of $V(D)$ can be partitioned into four disjoint vertex sets $W_1$, $W_2$, $W_3$ and $W_4$ such  that $|W_1|\approx|W_3|$ and $|W_2|\approx|W_4|$; here, $D[W_1]$ and $D[W_1]$ are almost complete, and $(W_2, W_4)$ is an almost complete bipartite digraph.

The proofs of Theorems \ref{song1} and \ref{song2} are completed in each extremal case through a detailed structural analysis.
\section{Non-extremal case: proofs of Theorems \ref{song1} and \ref{song2}}
Actually, we prove Theorems \ref{song1} and \ref{song2} in a stronger form. We show that if $D$ satisfies the conditions of Theorem \ref{song1}, then the following holds. For any sequence $l_1, \ldots, l_h$ with each $l_i\geq3$, if either $D$ is $\varepsilon^\prime$-stable and the sum of the $l_i$ less than $\alpha n$ is no more than $\beta n$, or $D$ is not $\varepsilon^\prime$-stable and $\max\{l_1, \ldots, l_h\}\leq(1/2+2\varepsilon^{1/2})n$, where $\alpha, \beta, \varepsilon^\prime$ and $\varepsilon$ are constants with $0<\alpha, \beta\ll\varepsilon^\prime\ll\varepsilon\ll1$, then $D$ contains a spanning $H$-subdivision, in which the paths have lengths $l_1, \ldots, l_h$, respectively, In addition, Theorem \ref{song2} holds if either $D$ is $\varepsilon^\prime$-stable and $f$ is any mapping from $V(H)$ to $V(D)$, or $D$ is not $\varepsilon^\prime$-stable.
\subsection{Absorbing and path-covering lemmas}
Let $C_0>0$ be a constant and $C$ be any integer with $C\geq C_0$, as shown in Theorems \ref{song1} and \ref{song2}. Parameters $\varepsilon^\prime, \gamma_1$ and $\gamma$ are chosen such that $$0<1/C\ll\gamma, \gamma_1\ll\varepsilon^\prime\ll1.$$
The following standard inequalities will be frequently used.
\begin{lemma}\cite{Janson}
Let $X$ be a random variable with the expectation $\mathbb{E}X$, and $a$ is any real number with $0<a<3/2$. Then the following statements hold.\\
$(1)$ $($Chernoff's inequality$)$ $\mathbb{P}(|X-\mathbb{E}X|>a\mathbb{E}X)<2e^{-\frac{a^2}{3}\mathbb{E}X}$.\\
$(2)$ $($Markov's inequality$)$ $\mathbb{P}(X\geq a)\leq\frac{\mathbb{E}X}{a}$.
\end{lemma}
For a vertex pair $(u, v)$ (possibly, $u=v$), we say that a $4$-path $z_1z_2z_3z_4$ \emph{absorbs} $(u, v)$ if $z_2u, vz_3\in A$. Also, a $4$-path is called an \emph{absorber} for $(u, v)$ if it absorbs $(u, v)$. This terminology reflects the fact that the $4$-path $z_1z_2z_3z_4$ can be extended by absorbing a path with end-vertices $u$ and $v$, resulting in a longer path with the same set of end-vertices. We say that a set
$F$ of disjoint subdigraphs absorbs a vertex set $U$ if for every vertex in $U$, there exists a $4$-path in
$F$ that absorbs it, and the absorbing $4$-paths for distinct vertices are distinct. Finally, for any vertex pair $(u, v)$ (possibly $u=v$) in $D$, we denote by $\mathcal{A}_{uv}$ the family of all $4$-paths absorbing $(u, v)$. Then, we can conclude that for any vertex pair of $D$, there are at least $\gamma n^4$ different $4$-paths to absorb it.

\begin{fact}\label{qqqw}
Let $D$ be an $n$-vertex digraph with $\delta^0(D)\geq(1/2-\gamma)n$, and $D$ is $\varepsilon^\prime$-stable. The parameters $\gamma$ and $\varepsilon^\prime$ satisfy $0<1/C\ll\gamma\ll\varepsilon^\prime\ll1$. Then for any vertex pair $(u, v)$ of $D$, we have that $|\mathcal{A}_{uv}|\geq\gamma n^4$.
\end{fact}
\begin{proof}
	Let $U_1=N^-(u)$ and $U_2=N^+(v)$. By the lower bound of $\delta^0(D)$, we have that $|U_i|\geq (1/2-\gamma)n$ for every $i\in[2]$. Since $D$ is $\varepsilon^\prime$-stable, we obtain that $e^+(U_1, U_2)>(\varepsilon^\prime n)^2$. Then again by the lower bound of $\delta^0(D)$, we can deduce that for any given arc $z_2z_3$ with the vertex $z_2\in U_1$ and $z_3\in U_2$, we have that $|N^-(z_2)\setminus\{u, v, z_3\}|\geq (1/2-\gamma)n-3$ and $|N^+(z_3)\setminus\{u, v, z_2\}|\geq (1/2-\gamma)n-3$. This implies that the number of $4$-paths $z_1z_2z_3z_4$ with $z_1\in N^-(z_2)\setminus\{u, v, z_3\}$ and $z_4\in N^+(z_3)\setminus\{u, v, z_1, z_2\}$ absorbing $(u, v)$ is at least
	\begin{equation*}
		\begin{split}
			(\varepsilon^\prime n)^2\cdot((1/2-\gamma)n-3)\cdot((1/2-\gamma)n-4)\geq(\varepsilon^\prime)^2n^4/5\geq \gamma n^4.
		\end{split}
	\end{equation*}
	This completes the proof of Fact \ref{qqqw}.
\end{proof}
Before presenting the absorption lemma, it is necessary to introduce the following useful lemmas. Lemma \ref{absor1} below states that for a digraph $D$ with sufficiently large minimum semi-degree, we can get a small set $\mathcal{F}$ of disjoint absorbers, satisfying that for any vertex pair $(u, v)$ in $D$, there are enough absorbers in $\mathcal{F}$ to absorb $(u, v)$. Further, Lemma \ref{absor2} below shows that we can partition this set $\mathcal{F}$ into $l$ disjoint parts such that in each part, there is at least one absorber for any vertex pair $(u, v)$, and that for each part, by using some vertices in $V(D-\mathcal{F})$ and the extremal condition with parameter $\varepsilon$ (EC($\varepsilon$)), we can get a path covering all absorbers of this part.
\begin{lemma}\label{absor1}
Let $D$ be an $n$-vertex digraph with $\delta^0(D)\geq (1/2-\gamma)n$, and $D$ is $\varepsilon^\prime$-stable. Then there exists a family $\mathcal{F}$ of at most $\gamma n$ disjoint absorbers
in $D$ such that for every vertex pair $(u, v)$, we have
$|\mathcal{A}_{uv}\cap\mathcal{F}|\geq\gamma^2n$.
\end{lemma}
\begin{proof}
Let $\gamma_1$ be a real number such that $\frac{\gamma_1^2}{2}\geq \gamma^2$ and $2\gamma_1^3\leq \gamma$. We select a family $\mathcal{F}^\prime$ of $4$-sets from $[V(D)]^4$ at random by including each of $n(n-1)(n-2)(n-3)\sim n^4$ of them independently with probability $\gamma_1^3 n^{-3}$ (some of the selected $4$-sets may not be absorbing at all). By Chernoff's inequality, since $\mathbb{E}|\mathcal{F}^\prime|=n^4\cdot\gamma_1^3 n^{-3}=\gamma_1^3n$, we have: $$\mathbb{P}(|\mathcal{F}^\prime|\geq2\gamma_1^3n)\leq\mathbb{P}(||\mathcal{F}^\prime|-\mathbb{E}|\mathcal{F}^\prime||\geq \gamma_1^3n)\leq2e^{-\frac{1}{3}\mathbb{E}|\mathcal{F}^\prime|}.$$
Similarly, for every vertex pair $(u, v)$, by Fact \ref{qqqw}, we get that $\mathbb{E}|\mathcal{A}_{uv}\cap\mathcal{F}^\prime|=\gamma n^4\cdot\gamma_1^3 n^{-3}=\gamma\gamma_1^3n$. So by Chernoff's inequality again, we have that $\mathbb{P}(|\mathcal{A}_{uv}\cap\mathcal{F}^\prime|\leq\gamma_1^4n/3)\leq\mathbb{P}(||\mathcal{A}_{uv}\cap\mathcal{F}^\prime|
-\mathbb{E}|\mathcal{A}_{uv}\cap\mathcal{F}^\prime||\geq \gamma\gamma_1^3n-\gamma_1^4n/3)\leq2e^{-\frac{1}{12}\mathbb{E}|\mathcal{A}_{uv}\cap\mathcal{F}^\prime|}$.
Thus, with probability $1-o(1)$, as $n\rightarrow \infty$:

\smallskip

\emph{$(1)$ $|\mathcal{F}^\prime|<2\gamma_1^3n\leq\gamma n$, and $|\mathcal{A}_{uv}\cap\mathcal{F}^\prime|>\gamma_1^4n/3$ for every vertex pair $(u, v)$.}

\smallskip

\noindent Clearly, the expected number of intersecting pairs of $4$-sets in $\mathcal{F}^\prime$ is at most
\begin{equation*}
\begin{split}
n^4\times (C_4^1n^3+C_4^2n^2+C_4^3n)\times(\gamma_1^3 n^{-3})^2\leq n^4\times4\times4\times n^3\times(\gamma_1^3 n^{-3})^2=16\gamma_1^6n,
\end{split}
\end{equation*}
so by Markov's inequality: with $X:=$ the number of intersecting pairs of $4$-sets in $\mathcal{F}^\prime$ and $a=17\gamma_1^6n$, we can get that $\mathbb{P}(X\geq 17\gamma_1^6n)\leq\frac{\mathbb{E}X}{a}=\frac{16\gamma_1^6n}{17\gamma_1^6n}=16/17$. This implies that

\smallskip

\emph{$(2)$ with probability at least $1/17$, as $n\rightarrow\infty$, there are at most $17\gamma_1^6n$ pairs of intersecting $4$-sets in $\mathcal{F}^\prime$.}

\smallskip

\noindent Hence, with positive probability, the family $\mathcal{F}^\prime$ satisfies the properties both $(1)$ and $(2)$, which implies that there exists one such family, and, for simplicity, we define this family to be $\mathcal{F}^{\prime\prime}$. From $\mathcal{F}^{\prime\prime}$ we delete all $4$-sets that intersect other $4$-sets, as well as all $4$-sets that are not absorbers, and denote by $\mathcal{F}$ the remaining subfamily. Clearly, by $(1)$ and $(2)$ again, we have that $|\mathcal{F}|\geq2\gamma_1^3n-17\gamma_1^6n\geq3\gamma_1^3n/2$. Also, the family $\mathcal{F}$ consists of disjoint absorbers and by $(2)$ and $(1)$, for every vertex pair $(u, v)$,
\begin{equation*}
\begin{split}
|\mathcal{A}_{uv}\cap\mathcal{F}|>\frac{\gamma_1^4n}{3}-2\cdot17\gamma_1^6n>\frac{\gamma_1^4n}{4}\geq \gamma^2n.
\end{split}
\end{equation*}
This completes the proof of the lemma.
\end{proof}
For two paths $P=a\cdots b$ and $Q=b\cdots d$ with $V(P)\cap V(Q)=\{b\}$, we denote the concatenated path as $P\circ Q$. This definition can be extended naturally to more than two paths. Further, the following conclusion holds.
\begin{lemma}\label{absor2}
Let $D$ be an $n$-vertex digraph with $\delta^0(D)\geq (1/2-\gamma)n$, and $D$ is $\varepsilon^\prime$-stable. Let $\mathcal{F}$ be the family with $|\mathcal{F}|=f$ obtained in Lemma \ref{absor1}. For any real $\beta$ with $\gamma^2<\beta<1-\gamma^2$ and any partition $f=f_1+\cdots+f_l$ with $\beta f<f_i<(1-\beta)f$ for each $i\in[l]$, there exists a partition $\mathcal{F}=\mathcal{F}_1\cup \cdots\cup\mathcal{F}_l$ with $|\mathcal{F}_i|=f_i$ such that

$(i)$ for any vertex pair $(u, v)$ of $D-V(\mathcal{F})$ and any $i\in[l]$, $\mathcal{F}_i$ contains at least one absorber for $(u, v)$.

$(ii)$ for each $i\in[l]$, we further define $\mathcal{F}_i=\{F_{i, 1}, \ldots, F_{i, f_i}\}$. Then there exists a path $L_i$ in $D$ of the form $L_i=F_{i, 1}\circ P_{i, 1}\circ F_{i, 2}\cdots\circ F_{i, f_i-1}\circ P_{i, f_i-1}\circ F_{i, f_i}$, where for each $j\in[f_i-1]$, the path $P_{i, j}$ has a length at most $3$, and all paths $L_i$ are disjoint.
\end{lemma}
\begin{proof}
Recall that the parameters $\gamma_1, \gamma$ and $\varepsilon^\prime$ satisfy $2\gamma_1^3<\gamma\ll\varepsilon^\prime$. In the following, we first prove (i), and secondly give the proof of (ii).

\smallskip

We prove that there exists a partition $\mathcal{F}=\mathcal{F}_1\cup \cdots\cup\mathcal{F}_l$ satisfying our conditions with high probability. Let $F_{uv}^j$ be a random variable that calculates the number of absorbers of $(u, v)$ in $\mathcal{F}_j$ for $j\in[l]$. Since Lemma \ref{absor1} gives $|\mathcal{A}_{uv}\cap\mathcal{F}|\geq\gamma^2n$ and $|\mathcal{F}|<\gamma n$, it is not hard to see that
\begin{equation*}
\begin{split}
\mathbb{E}F_{uv}^j=\frac{|\mathcal{A}_{uv}\cap\mathcal{F}|}{|\mathcal{F}|}\cdot|\mathcal{F}_j|
\geq\frac{\gamma^2n}{\gamma n}\cdot |\mathcal{F}_j|=\gamma |\mathcal{F}_j|,
\end{split}
\end{equation*}

Then by Chernoff's bound with $a=\frac{1}{2}$, for any $j\in[l]$, due to $\beta f\leq|\mathcal{F}_j|\leq(1-\beta)f$ and $f\geq \gamma^2n$, we obtain that
\begin{align}\label{eee}
\mathbb{P}\left(F_{uv}^j<\frac{\mathbb{E}F_{uv}^j}{2}\right)\leq\mathbb{P}\left(|F_{uv}^j-\mathbb{E}F_{uv}^j|
>\frac{\mathbb{E}F_{uv}^j}{2}\right)
<2e^{-\frac{\mathbb{E}F_{uv}^j}{12}}\leq2e^{-\frac{\gamma|\mathcal{F}_j|}{12}}.
\end{align}
Notice that there are less than $n^2$ vertex pairs in $D-V(\mathcal{F})$. Since $|\mathcal{F}|\geq|\mathcal{A}_{uv}\cap\mathcal{F}|\geq\gamma^2n$, we have that $|\mathcal{F}_j|\geq \beta f\geq\beta\gamma^2n$. Then by (\ref{eee}), for $j\in[l]$, we get that
\begin{equation*}
\begin{split}
\sum_{\{u, v\}\subseteq V(D-\mathcal{F})}\mathbb{P}\left(F_{uv}^j<\frac{\mathbb{E}F_{uv}^j}{2}\right)
<2n^2e^{-\frac{\gamma|\mathcal{F}_j|}{12}}\leq 2n^2e^{-\frac{\beta\gamma^3n}{12}}\rightarrow 0, \mbox{as} \ n\rightarrow\infty.
\end{split}
\end{equation*}
Therefore with high probability, we obtain that $F_{uv}^j\geq\frac{\mathbb{E}F_{uv}^j}{2}\geq\frac{\gamma|\mathcal{F}_j|}{2}>0$ for any $j\in[l]$. Recapping the definition of $F_{uv}^j$, this implies that we may always find $l$ subsets $\mathcal{F}_1, \ldots, \mathcal{F}_l$ of absorbers in $\mathcal{F}$ such that for any vertex pair $(u, v)$ of $D-V(\mathcal{F})$ and any $i\in[l]$, the family $\mathcal{F}_i$ contains at least one absorber of $(u, v)$. This implies that (i) holds.

\smallskip

(ii) Let $\mathcal{F}_1\cup \cdots\cup\mathcal{F}_l$ be a partition of $\mathcal{F}$ satisfying (i). For any index $i\in[l]$, we have that $|\mathcal{F}_i|=f_i$. In the following, we will prove (ii) by induction on $k$ that for each $k\in[f_i]$, there exists a path $S_k$ in $D$ of the form $S_1=F_{i, 1}$ and for $k\geq2$,
\begin{equation*}\label{sk}
\begin{split}
S_k=F_{i, 1}\circ P_{i, 1}\cdots\circ F_{i, k-1}\circ P_{i, k-1}\circ F_{i, k},
\end{split}
\end{equation*}
where each of the paths $P_{i, 1}, \ldots, P_{i, k-1}$ has the length at most $3$. Note that $L_i=S_{f_i}$.

It is obvious for the case $k=1$. Assume that the equation above is true for some $k-1\in[f_i-1]$. Moreover, we suppose that the terminal of $F_{i, k-1}$ is $b$, and the initial of $F_{i, k}$ is $a$.
Denote by $D_{k-1}$ the subdigraph induced by the vertex set $V_{k-1}=V(D)\setminus V((S_{k-1}-b)\cup(\mathcal{F}-a))$ in $D$. By the proof of Lemma \ref{absor1}, we have that $|\mathcal{F}|\leq2\gamma_1^3 n$, and then
$|V(S_{k-1}\cup \mathcal{F})|<|\mathcal{F}|(4+4)=8f\leq8\cdot2\gamma_1^3 n<\gamma n$. This implies that
$|N^+_{D_{k-1}}(b)|, |N^-_{D_{k-1}}(a)|\geq\delta^0(D)-|V(S_{k-1}\cup \mathcal{F})|\geq(1/2-2\gamma)n>(1/2-\varepsilon^\prime)n$. Then by the extremal condition with $(U_1, U_2)_{EC(\varepsilon^\prime)}=(N^+_{D_{k-1}}(b), N^-_{D_{k-1}}(a))$, there is a path $P_{i, k-1}$ from $b$ to $a$ of length at most $3$ in $D_{k-1}$ such that it connects paths $S_{k-1}$ and $F_{i, k}$. Note that $V(P_{i, k-1})\setminus\{a, b\}$ is disjoint from $V(\mathcal{F}\cup S_{k-1})$, and so the desired path
\begin{equation*}
\begin{split}
S_{k}=S_{k-1}\circ P_{i, k-1}\circ F_{i, k}.
\end{split}
\end{equation*}
Hence, the proof of $(ii)$ is completed.
\end{proof}
By using Lemmas \ref{absor1} and \ref{absor2}, we will now present the following absorption lemma.
\begin{lemma} \emph{(}Absorbing Lemma\emph{)} \label{absorbing lemma2}
Suppose that $C>0$ is an integer and $\alpha, \beta$ are two constants with $0<\alpha, \beta\ll1$. Let $n=n_1+\cdots+n_m$ with $n\geq Chm$ and $\sum_{n_i<\alpha n}n_i\leq\beta n$. For any digraph $H$ with the vertex set $V(H)=\{v_1, \ldots, v_s\}$ and $h$ arcs, and $\delta(H)\geq1$, if $D$ is a digraph with the order $n$ and $\delta^0(D)\geq n/2$, and $D$ is $\varepsilon^\prime$-stable, then the following statements hold.

\smallskip

$(i)$ For any set $\mathcal{N}=\{l_1, \cdots, l_h\}$ satisfying that the sum of the $l_i$ less than $\alpha n$ is not more than $\beta n$, the digraph $D$ contains an $H$-subdivision, called $H^\prime$, with $|V(H^\prime)|\leq\gamma n$ and in which the length of internally disjoint mapped paths is at most $l_1, \ldots, l_h$, respectively.

\smallskip

$(ii)$ The digraph $D$ contains $m$ disjoint $H$-subdivisions, called $H^\prime_1, \ldots, H^\prime_m$, such that $|\bigcup_{i=1}^mV(H^\prime_i)|\leq5\gamma n$, and for each $i\in[m]$, we have that $|V(H^\prime_i)|\leq n_i$.

\smallskip

$(iii)$ Finally, we define the digraph $Q$ as follows:
\begin{align*}
\begin{split}
Q=\left \{
\begin{array}{ll}
H^\prime, &for\ (i),\\
H^\prime_1\cup\cdots\cup H^\prime_m, &for\ (ii).
\end{array}
\right.
\end{split}
\end{align*}
Then, for every subset $U\subset V(D-Q)$ of cardinality at most $\gamma^2n$, $Q$ can absorb the set $U$ and maintain the same branch-vertices as them.
\end{lemma}
\begin{proof}
Suppose that $\gamma$ is a real number satisfying that $$0<1/C\ll\alpha, \beta<\gamma\ll\varepsilon^\prime.$$

\smallskip

(i) Without loss of generality, we assume that $l_1\geq l_2\geq\cdots\geq l_h$. Take $t$ from $[h]$ such that $t$ is the largest subscript satisfying $n_t\geq\alpha n$. Note that $t$ is an absolute constant independent of $n$ since $t\leq\frac{n}{\alpha n}=\frac{1}{\alpha}$.  We define $W=f(V(H)):=\{u_1, \ldots, u_s\}$ and let $D^\prime=D-W$. Clearly, $\delta^0(D^\prime)\geq n/2-s\geq n/2-2h\geq(\frac{1}{2}-\frac{1}{Cm})n$ since $n\geq Chm$.

\smallskip

On the one hand, in the desired $H$-subdivision $H^\prime$, if there is a path of length $l_{t+1}<\alpha n$, say from $u_1$ to $u_2$, then we can greedily get this path by doing the following. We choose a vertex $w_3\in N_{D^\prime}^+(u_1)$ and $w_{i+1}\in N^+_{D^\prime}(w_i)\setminus \{w_3, \ldots, w_i\}$ for all $i\in\{3, \ldots, l_{t+1}-3\}$. Further, because $l_{t+1}+\cdots+l_h\leq \beta n$ and $\delta^0(D^\prime)\geq(\frac{1}{2}-\frac{1}{Cm})n$, we have that $|N^+_{D^\prime}(w_{l_{t+1}-2})\setminus\{u_1, w_3, \ldots, w_{l_{t+1}-2}, u_2\}|, |N^-_{D^\prime}(u_2)\setminus\{u_1, w_3, \ldots, w_{l_{t+1}-2}, u_2\}|\geq(1/2-\varepsilon^\prime)n$. Then this implies that there is an arc from the former to the latter, since $D$ does not satisfy EC($\varepsilon$) and $D^\prime\subseteq D$. Using a similar process as shown below, we can obtain all paths $P_{t+1}, \ldots, P_h$ of length $l_{t+1}, \ldots, l_h$, respectively, of the desired $H$-subdivision $H^\prime$.

\smallskip

In the following, let $D^{\prime\prime}=D^\prime-\bigcup_{j=t+1}^hP_j$. Clearly, $\delta^0(D^{\prime\prime})\geq(1/2-\beta)n$. Then by Lemma \ref{absor1}, there is a family $\mathcal{F}$ of at most $\gamma n$ disjoint, absorbing $4$-paths in $D^{\prime\prime}$ such that for every vertex pair $(u, v)$, we have that $|\mathcal{A}_{uv}\cap\mathcal{F}|\geq\gamma^2n$. Then by using Lemma \ref{absor2}, we can partition $\mathcal{F}$ into $t$ distinct parts, say $\mathcal{F}_1, \ldots, \mathcal{F}_t$, such that $|V(\mathcal{F}_i)|\leq l_i-6$ for each $i\in[t]$, and in each part, there is an absorber for any vertex pair $(u, v)$ of $V(D^{\prime\prime}-\mathcal{F})$, and Lemma \ref{absor2} also depicts that we can connect all absorbing paths in $\mathcal{F}_i$ into a path $L_i$ of length no more than $l_i-6$. Further, these paths $L_1, \ldots, L_t$ are disjoint and $|\bigcup_{i=1}^tV(L_i)|\leq\gamma n$. Without loss of generality, write $L_i=a_i\cdots b_i$ for any $i\in[t]$, and suppose that in the desired $H$-subdivision $H^\prime$, the path from $u_i$ to $u_i^\prime$ has the length at most $l_i$. Define $D^0=D^{\prime\prime}-\bigcup_{j=1}^tL_j$, and let $R_i=N_{D^0}^+(u_i)$ and $S_i=N_{D^0}^-(a_i)$ for each $i\in[t]$. Then we have that $|R_i|, |S_i|\geq\delta^0(D^{\prime\prime})-\gamma n\geq(1/2-\varepsilon^\prime)n$, which implies that there exists an arc $x_ix_i^\prime$ from $R_i$ to $S_i$, since $D^0\subset D$ does not meet EC($\varepsilon$). Similarly, we can prove that there is an arc $y_iy_i^\prime$ from $N_{D^0}^+(b_i)$ to $N_{D^0}^+(u_i^\prime)$ for each $i\in[t]$. Then we get a path $u_ix_ix_i^\prime L_iy_iy_i^\prime u_i^\prime$. Hence we obtain all paths of length at most $l_1, \ldots, l_t$, respectively.

\smallskip

To sum up, we get an $H$-subdivision $H^\prime$ in $D$, in which the length of internally disjoint mapped paths
is at most $l_1, \ldots, l_h$, respectively. Clearly, $|V(H^\prime)|\leq\beta n+4h+|V(\mathcal{F})|\leq\gamma n$. This completes the proof of (i).

\medskip

(ii) In $D$, we randomly take the disjoint mapping vertex sets of $V(H)$ to be $f_1(V(H))=W_1:=\{v_1^1, v_2^1, \ldots, v_s^1\}$, \ldots, $f_m(V(H))=W_m:=\{v_1^m, v_2^m, \ldots, v_s^m\}$, respectively. Let $D^\prime=D-\cup_{i=1}^m W_i$, and clearly, since $n\geq Chm$,
\begin{equation*}
\begin{split}
\delta^0(D^\prime)\geq \frac{n}{2}-2mh\geq\frac{n}{2}-\frac{2n}{C}\geq(1/2-\beta)n.
\end{split}
\end{equation*}
Without loss of generality, we assume that $n_1\geq\cdots\geq n_t\geq\alpha n>n_{t+1}\geq\cdots\geq n_m$. In the following, we will first get $t$ disjoint $H$-subdivisions $H_1^\prime, \ldots, H_t^\prime$ of order at most $n_1, \ldots, n_t$, respectively, such that $|\bigcup_{i=1}^tV(H_i^\prime)|\leq4\gamma n$. Secondly, we will prove that in the remaining digraph, there are $m-t$ disjoint $H$-subdivisions, that the order of each such $H$-subdivision is $n_{t+1}, \ldots, n_m$, respectively.

\smallskip

Firstly, similar to the proof of (i), Lemma \ref{absor1} gives a family $\mathcal{F}$ of at most $\gamma n$ disjoint absorbers in $D^\prime$ such that for every vertex pair $(u, v)$, it satisfies that
$|\mathcal{A}_{uv}\cap\mathcal{F}|\geq\gamma^2n$. Then applying Lemma \ref{absor2}, we can partition the family $\mathcal{F}$ into $t$ disjoint subfamily, say $\mathcal{F}_1, \ldots, \mathcal{F}_t$, such that for each $i\in[t]$, $|V(\mathcal{F}_i)|<n_i$, and further, all $4$-paths in $\mathcal{F}_{i}$ can be connected into a path $L_i$ with $|V(L_i)|<n_i$. Without loss of generality, assume that $L_i=a_i\cdots b_i$ for every $i\in[t]$. For any $i\in[t]$, in the desired $H$-subdivision $H^\prime_i$, the path, say from $v_1^i$ to $v_2^i$, can be gotten by the following strategy as below. We define $D^{\prime\prime}=D^{\prime}-\bigcup_{j=1}^tL_j$, and let $R_i=N_{D^{\prime\prime}}^+(v_1^i)$ and $S_i=N_{D^{\prime\prime}}^-(a_i)$ for each $i\in[t]$. Then we have that $|R_i|, |S_i|\geq\delta^0(D^{\prime\prime})\geq(1/2-\beta)n-\gamma n\geq(1/2-2\gamma)n$, which implies that there exists an arc $x_1^ix_2^i$ from $R_i$ to $S_i$, since $D^{\prime\prime}\subset D$ does not EC($\varepsilon$). Similarly, we can prove that there is an arc $y_1^iy_2^i$ from $N_{D^{\prime\prime}}^+(b_i)$ to $N_{D^{\prime\prime}}^+(v_2^i)$ for each $i\in[t]$. Then we get a path $v_1^ix_1^ix_2^iL_iy_1^iy_2^iv_2^i$ for any $i\in[t]$. This can be done because $\delta^0(D^{\prime\prime})-4t\geq(1/2-2\gamma)n-\frac{n}{Ch}\geq(1/2-3\gamma)n>0$. Let $D^3$ be the remaining digraph, and clearly $\delta^0(D^3)\geq(1/2-3\gamma)n$. Further, for the other paths in the desired $H$-subdivision $H_i^\prime$ with $i\in[t]$, we can obtain these paths of length at most $3$. This is because by $|N_{D^3}^\sigma(v_j^i)|\geq\delta^0(D^3)>\delta^0(D^3)-2hm\geq(1/2-3\gamma)n-\frac{n}{C}\geq(1/2-7\gamma/2)n$ for any $\sigma\in\{+, -\}$, $i\in[m]$ and $j\in[h]$, and $D^3\subset D$ does not satisfy EC($\varepsilon$), we get that there is at least $(\varepsilon^\prime n)^2$ arcs from $N_{D^3}^\sigma(v_j^i)$ to $N_{D^3}^{-\sigma}(v_k^l)$ with $j, k\in[h]$ and $i, l\in[m]$ and $(U_1, U_2)_{EC(\varepsilon)}=(N_{D^3}^\sigma(v_j^i), N_{D^3}^{-\sigma}(v_k^l))$. This suggests that we obtain $t$ disjoint $H$-subdivisions $H_1^\prime, \ldots, H_t^\prime$ with $|\bigcup_{i=1}^tV(H_i^\prime)|\leq7\gamma n/2$, and in which, the order of $H_1^\prime, \ldots, H_t^\prime$ is at most $n_1, \ldots, n_t$, respectively.

\smallskip

Let $D^4=D^3-\bigcup_{i=1}^tH_i^\prime$, and obviously $\delta^0(D^4)\geq(1/2-7\gamma/2)n$. In the following, we will prove that $D^4$ contains $m-t$ disjoint $H$-subdivisions with the order $n_{t+1}, \ldots, n_m$, respectively. Similar to the proof at the end of the previous paragraph, we can get that for each $i\in\{t+1, \ldots, m\}$, there are $h-1$ disjoint paths of length at most $3$ as desired in the $H$-subdivision $H_i^\prime$, since for any vertex $v_l^k$ with $l\in[h], k\in[m]\setminus[t]$, we have that
\begin{equation*}
\begin{split}
|N_{D^4}^\sigma(v_l^k)|\geq\delta^0(D^4)>\delta^0(D^4)-2hm\geq (1/2-7\gamma/2)n-2n/C\geq(1/2-4\gamma)n.
\end{split}
\end{equation*}
Finally, in every promised $H$-subdivision $H_i^\prime$ with $i\in[m]\setminus[t]$, only one path is missing, without loss of generality, say from $v_{h-1}^i$ to $v_h^i$ order $n_i^\prime$. Let $D^5$ be the remaining digraph, and then clearly, $\delta^0(D^5)\geq\delta^0(D^4)-2hm\geq(1/2-4\gamma)n$ since $n\geq Chm$ and $1/C\ll \gamma$. Then we can greedily get these paths by doing the following. Let $i$ be any index with $i\in[m]\setminus[t]$. We choose a vertex $w_3\in N_{D^5}^+(v_{h-1}^i)$ and $w_{j+1}\in N^+_{D^5}(w_j)\setminus \{w_3, \ldots, w_j\}$ for all $j\in\{3, \ldots, n_i^\prime\}$. Further, because $n_{t+1}^\prime+\cdots+n_m^\prime\leq \beta n$ and $\delta^0(D^5)\geq(1/2-4\gamma)n$, we have that $|N^+_{D^5}(w_{n_{i}^\prime-2})\setminus\{v_{h-1}^i, w_3, \ldots, w_{n_{i}^\prime-2}, v_{h}^i\}|, |N^-_{D^\prime}(v_{h}^i)\setminus\{v_{h-1}^i, w_3, \ldots, w_{n_{i}^\prime-2}, v_{h}^i\}|\geq(1/2-4\gamma)n$. Let $R$ be the previous set and $S$ be the subsequent set. Then this implies that there is an arc from $R$ to $S$, since $D$ does not satisfy EC($\varepsilon$) and $D^5\subseteq D$. Using a similar process as shown below, we can obtain all paths of order $n_{t+1}^\prime, \ldots, n_m^\prime$, of the desired $H$-subdivisions $H_{t+1}^\prime, \ldots, H_{m}^\prime$, respectively, due to
\begin{equation}\label{lllq}
\begin{split}
\delta^0(D^5)-(n_{t+1}^\prime+\cdots+n_m^\prime)\geq(1/2-4\gamma)n-\beta n\geq(1/2-5\gamma)n.
\end{split}
\end{equation}
Hence we get $m-t$ disjoint $H$-subdivisions of order $n_{t+1}, \ldots, n_m$, respectively.

\smallskip

So we get $m$ disjoint $H$-subdivisions $H_1^\prime, \ldots, H_m^\prime$ such that $|V(H_i^\prime)|<n_i$ for $n_i\geq\alpha n$ and $|V(H_i^\prime)|=n_i$ for $n_i<\alpha n$. Also, by $(\ref{lllq})$, we have that $|\bigcup_{i=1}^mV(H_i^\prime)|\leq5\gamma n$. This implies that (ii) is true.

\smallskip

(iii) In the following, we will prove that $Q$ captures the the absorption property. By Lemma \ref{absor2} and the constructions of $H^\prime_1, \ldots, H^\prime_m$, for each $i\in[t]$ ($t$ is defined in $(ii)$. Also, maybe $t=0$ and in the case, it is $H^\prime$), there exists an absorber in $H^\prime_i$ ($i\in[t]$) for every vertex pair $(u, v)$ of $D-Q$. Let $U\subseteq V(D-Q)$ be a vertex set with the cardinality at most $\gamma^2n$. By Lemma \ref{absor1}, for every vertex $u\in U$, we have $|\mathcal{A}_{u}\cap Q|>\gamma^2n$. Thus, we can insert all the vertices of $U$ with $|U|\leq\gamma^2n$ into $Q$ one by one, and each time using a fresh absorber, and still maintain the same branch-vertices of $H^\prime_1, \ldots, H^\prime_t$. So (iii) holds.

\smallskip

Hence we complete the proof of this lemma.
\end{proof}
Let $\nu$ and $\tau$ be real numbers with $0<\nu\leq\tau<1$. Suppose that $D$ is a digraph and $S$ is a vertex subset of $V(D)$. The \emph{$\nu$-robust out-neighbourhood $RN^+_{\nu, D}(S)$ of $S$} is defined as the set of all vertices $x$ in $D$ that have at least $\nu|V(D)|$ in-neighbourhoods in $S$. Moreover, the digraph $D$ is called a \emph{robust $(\nu, \tau)$-outexpander} if
$|RN^+_{\nu, D}(S)|\geq|S|+\nu|V(D)|$ for all $S\subseteq V(D)$ with $\tau|V(D)|<|S|<(1-\tau)|V(D)|$.  K\"{u}hn, Osthus and Treglown \cite{Kuhn1} depicted that for a positive integer $n_0$ and  $1/n_0\ll\nu\leq\tau\ll\xi< 1$, if $D$ is a digraph on $n\geq n_0$ vertices with $\delta^0(D)\geq\xi n$ and is a robust $(\nu, \tau)$-outexpander, then $D$ contains a Hamiltonian cycle.

\smallskip

In particular, if $D$ is a digraph with $\delta^0(D)\geq(1/2-\varepsilon^\prime)n$, where $0<\varepsilon^\prime\ll\varepsilon\ll1$, and $D$ is $\varepsilon^\prime$-stable with parameter $\varepsilon$, then we will prove that $D$ is a robust $(\nu, \tau)$-outexpander, and so the following path-covering lemma holds. Since a detailed proof in our paper \cite{Wang} has been given, we will omit its proof here.

\begin{lemma}\emph{(}Path-Covering Lemma, [\citenum{Wang}, Lemma~3.8]\emph{)}\label{Path-Cover Lemma}
Suppose that $D$ is a digraph on $n$ vertices with $\delta^0(D)\geq(1/2-5\gamma)n$, and is $\varepsilon^\prime$-stable, where $0<\gamma\ll\varepsilon^\prime\ll1$. Then $D$ contains a Hamiltonian path.
\end{lemma}
\subsection{Completion of Theorems \ref{song1} and \ref{song2}}
In this subsection, we always assume that $H$ is any digraph with $\delta(H)\geq1$, and $D$ is an $n$-vertex digraph $\delta^0(D)\geq n/2$, and $D$ is $\varepsilon^\prime$-stable. Let $C>0$ be a constant, and take parameters $\alpha, \beta$ and $\gamma$ satisfy $0<1/C\ll\alpha, \beta<\gamma\ll\varepsilon^\prime\ll1$. We first give the proof of Theorem \ref{song1}, and then we will prove Theorem \ref{song2}.

\smallskip

\textbf{Proof of Theorem \ref{song1}.}
Let $\mathcal{N}=\{l_1, \ldots, l_h\}$ be any integer set with $l_1+\cdots+l_h<n$. Without loss of generality, suppose that $l_1\geq\cdots\geq l_t\geq\alpha n>l_{t+1}\geq \cdots\geq l_h$, and the sum of the $l_i$ less than $\alpha n$ is not more than $\beta n$. By Lemma \ref{absorbing lemma2}(i), we can obtain an $H$-subdivision in $D$, called $H^\prime$, with $|V(H^\prime)|\leq\gamma n$, and in $H^\prime$, the paths of length less than $\alpha n$ have reached the lengths as required in the desired $H$-subdivision. However, for the paths of length at least $\alpha n$, in $H^\prime$, there are corresponding ``long'' paths, defined as $Q_1, \ldots, Q_t$, of length no more than $l_1, \ldots, l_t$, respectively, and each of these long paths contains an absorber of any vertex pair $(u, v)$ in $D-H^\prime$. With the aid of Lemma \ref{Path-Cover Lemma}, we can divide the Hamiltonian path in $D-H^\prime$ into $t$ disjoint subpaths, defined $R_1, \ldots, R_t$, such that for each $\in[t]$, $P_i=c_i\cdots d_i$ and the cardinality of $V(Q_i\cup R_i)$ is $l_i+1$. Since the path $Q_i$ contains an absorber of the pair $(c_i, d_i)$ for any $i\in[t]$, the path $Q_i$ can absorb $R_i$ to get a path with the same initial and terminal as $Q_i$ and the length $l_i$ as desired. Therefore, we get a spanning $H$-subdivision, in which, the length of paths is $l_1, \ldots, l_h$, respectively. This proves Theorem \ref{song1} for the non-extremal case.\hfill $\Box$

\medskip

\textbf{Proof of Theorem \ref{song2}.}
Let $n=n_1+\cdots+n_m$ be any integer partition of $n$. Without loss of generality, we assume that $n_1\geq \cdots\geq n_t\geq\alpha n>n_{t+1}\geq\cdots\geq n_m$, and the sum of the $n_i$ less than $\alpha n$ is no more than $\beta n$. With the help of Lemma \ref{absorbing lemma2}(ii), we can obtain $m$ disjoint $H$-subdivisions $H^\prime_1, \ldots, H^\prime_m$, such that and for each $i\in[t]$, we have $|V(H^\prime_i)|\leq n_i$, and for any $i\in[m]\setminus[t]$, we have that $|V(H^\prime_i)|=n_i$, and $|\bigcup_{i=1}^mV(H^\prime_i)|\leq4\gamma n$. Also, for every vertex pair $(u, v)$ of $D-\bigcup_{i=1}^m H^\prime_i$, there exists an absorber in $H^\prime_i$ ($i\in[t]$) for it.
Let $D^{0}=D-\bigcup_{i=1}^m H_{i}^\prime$, and then $\delta^0(D^0)\geq(1/2-4\gamma)n$. Then by Lemma \ref{Path-Cover Lemma}, we know that $D^{0}$ contains a Hamiltonian path, and then we divide this Hamiltonian path into $t$ disjoint paths $P_1, \ldots, P_t$, of order $n_1-V(H^\prime_1), \ldots, n_t-V(H^\prime_t)$, respectively. Without loss of generality, for any $j\in[t]$, we assume that $P_j=c_j\cdots d_j$. By Lemma \ref{absorbing lemma2}(iii), we get that for every $j\in[t]$, the subdigraph $H^\prime_j$ contains an absorber for $(c_j, d_j)$, and so $H^\prime_j\cup P_j$ is an $H$-subdivision of order $n_j$ with the same branch-vertices as $H^\prime_j$. So far, we have obtained $m$ disjoint $H$-subdivisions, of order $n_1, \ldots, n_m$, respectively.
Hence, this suggests that Theorem \ref{song2} holds for the non-extremal case when $D$ is $\varepsilon^\prime$-stable.
\hfill $\Box$
\section{Extremal case: proofs of Theorems \ref{song1} and \ref{song2}}
In this section, we give the unified proofs of main theorems when $D$ is not $\varepsilon^\prime$-stable. Parameters $\varepsilon^\prime$, $\varepsilon_1$ and $\varepsilon$ are chosen with $0<\varepsilon^\prime\ll\varepsilon_1\ll\varepsilon\ll1$. For any $X, Y\subseteq V(D)$, we denote by $(X, Y)$ the bipartite digraph with vertex set $X\cup Y$ and arc set $\{xy, yx\in A(D)| x\in X, y\in Y\}$.
We define the \emph{bidirectional neighbourhood} of a vertex $x$ in $D$ to be $BN(x)=\{y: xy, yx\in A(D)\}$, and the \emph{bidirectional semi-degree} of $x$ in $D$, defined as $b(x)$, is the cardinality of $BN(x)$, that is $b(x)=|BN(x)|$. Also, for a vertex subset $U$ of $D$, we denote $b_U(x)=|BN(x)\cap U|$. We also say two vertex sets $W_1$ and $W_2$ to be \emph{approximately equal}, written as $|W_1|\approx|W_2|$, if $|W_1|-|W_2|=\pm 2\varepsilon^\prime n$.

\subsection{Extremal cases and properties}
We now define the following three extremal cases that will occur when $D$ satisfies the extremal condition with parameter $\varepsilon$ (EC($\varepsilon$)).

\smallskip

\noindent \textbf{(1) Extremal Case $\mathbf{1}$ with parameter $\mathbf{\varepsilon}$ (EC1($\mathbf{\varepsilon}$)):} We can partition $V(D)$ into three disjoint vertex sets $W_1$, $W_2$ and $W_3$, such that for every $i\in[2]$, $|W_i|=(1/2-\varepsilon)n\pm\sqrt{10\varepsilon}|W_i|$, and $|W_3|\leq2\sqrt{10\varepsilon}\cdot\max\{|W_1|, |W_2|\}$, and the following properties hold. (See Figure \ref{fig1}$(a)$)

\smallskip

$(1.1)$ For each $i\in[2]$, there exists a subset $W_i^\prime$ in $W_i$ with $|W_i^\prime|\leq 10\sqrt{\varepsilon}|W_i|$ such that for every vertex $w$ in $W_i\setminus W_i^\prime$, $\delta^0_{W_i}(w)\geq (1-10\sqrt{\varepsilon})|W_i|$, and for every vertex $w$ in $W_i^\prime$, $\delta^0_{W_i}(w)\geq\varepsilon^{1/3}|W_i|/2$. We also say that $D[W_i]$ is \emph{$\varepsilon$-almost complete}, for each $i\in[2]$.

\smallskip

$(1.2)$ For any vertex $w\in W_3$, we have that either $d^+_{W_1}(w), d^-_{W_2}(w)>(1/2-\varepsilon^{1/3})n$, or $d^-_{W_1}(w), d^+_{W_2}(w)>(1/2-\varepsilon^{1/3})n$, and $\delta^0_{W_i}(w)\leq\varepsilon^{1/3}n/2$ for each $i\in[2]$.

\begin{figure}[h]
\centering
\scriptsize
\begin{tabular}{ccc}\label{7}
\includegraphics[width=4.35cm]{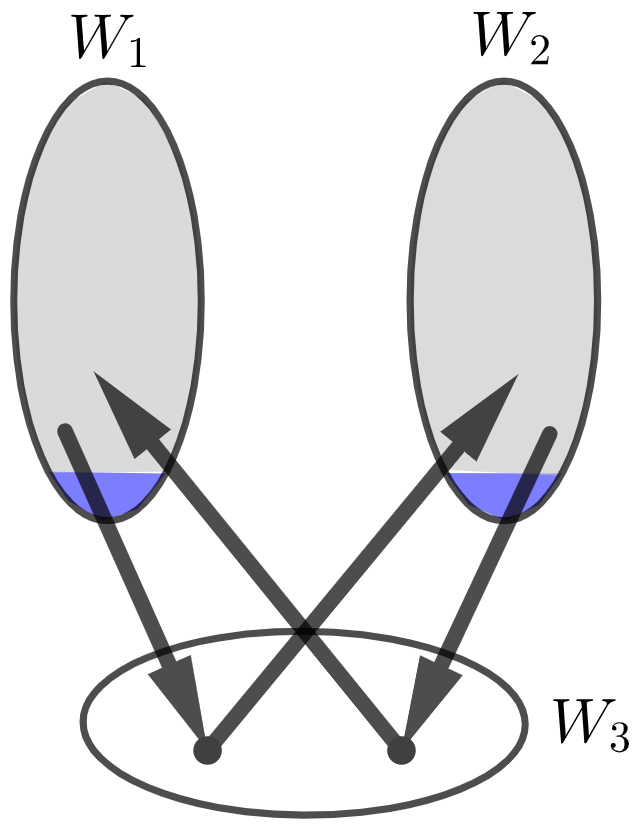}&\includegraphics[width=4.35cm]{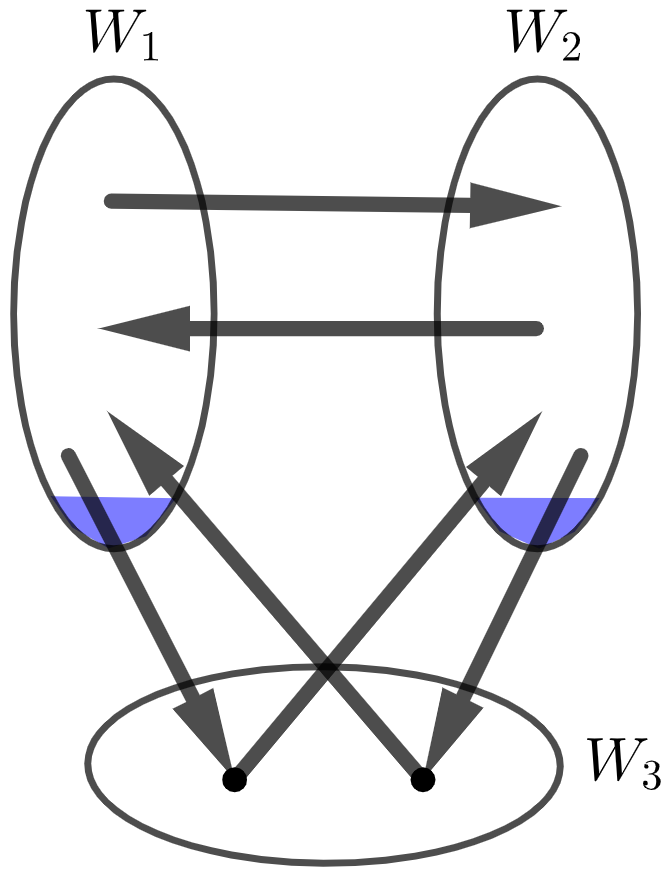}&\includegraphics[width=4.9cm]{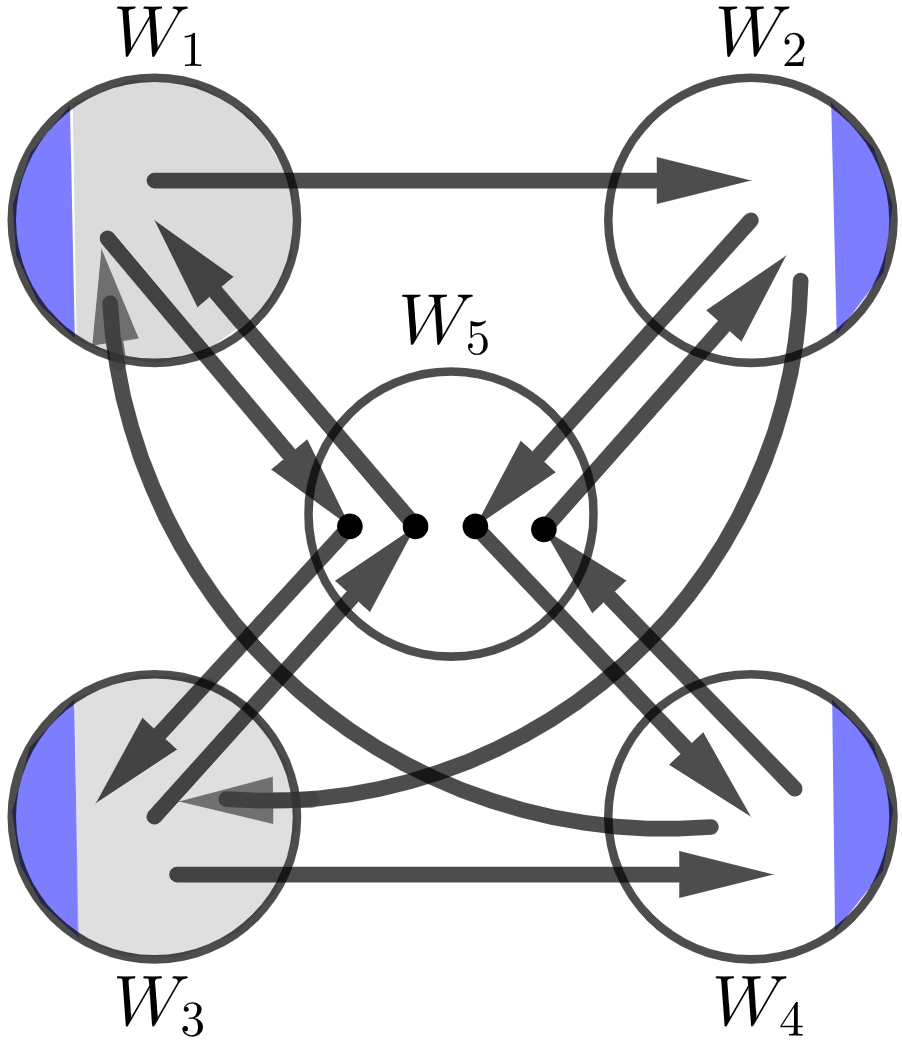}\\
(a) The extremal case $1$ (\textbf{EC1}). & (b) The extremal case $2$ (\textbf{EC2}). & (c) The extremal case $3$ (\textbf{EC3}).
\end{tabular}
\caption{The extremal cases $1$-$3$. Note that in this figure, a thick arrow pointing between two vertex sets indicates that the induced digraph by them is $\varepsilon$-almost one-way complete, and the gray shading indicates that the digraph induced by this vertex set is $\varepsilon$-almost complete. Also, the blue shaded areas indicate the exceptional vertices in the corresponding vertex sets.}
\label{fig1}
\vspace{-0.5em}
\end{figure}

\bigskip

\noindent \textbf{(2) Extremal Case $\mathbf{2}$ with parameter $\mathbf{\varepsilon}$ (EC2($\mathbf{\varepsilon}$)):} We can partition $V(D)$ into three disjoint vertex sets $W_1$, $W_2$ and $W_3$ such that for every $i\in[2]$, $|W_i|=(1/2-\varepsilon)n\pm\sqrt{10\varepsilon}|W_i|$, and $|W_3|\leq2\sqrt{10\varepsilon}\cdot\max\{|W_1|, |W_2|\}$, and the following hold. (See Figure \ref{fig1}$(b)$)

\smallskip

$(2.1)$ For every $i\in[2]$, apart from at most $10\sqrt{10\varepsilon}|W_i|$ exceptional vertices, all vertices in $W_i$ have bidirectional semi-degrees at least $(1-10\sqrt{\varepsilon})|W_{3-i}|$ in $W_{3-i}$, and the semi-degrees of these exceptional vertices are at least $\varepsilon^{1/3}|W_{3-i}|/8$ in $W_{3-i}$. In this case, we also call $(W_1, W_2)$ to be a \emph{$\varepsilon$-almost complete bipartite digraph}.

\smallskip

$(2.2)$ For every vertex $w\in W_3$, we have that either $d^+_{W_1}(w), d^-_{w_2}(w)>(1/2-\varepsilon^{1/3})n$, or $d^-_{W_1}(w), d^+_{W_2}(w)>(1/2-\varepsilon^{1/3})n$, and $\delta^0_{W_i}(w)\leq2\varepsilon^{1/3}n$ for each $i\in[2]$.

\bigskip

\noindent \textbf{(3) Extremal Case $\mathbf{3}$ with parameters $\mathbf{\varepsilon}$ and $\mathbf{\varepsilon_1}$ (EC3($\mathbf{\varepsilon}$, $\mathbf{\varepsilon_1}$)):} We can partition $V(D)$ into five disjoint vertex sets $W_1, W_2, W_3$, $W_4$ and $W_5$ in $D$ with $|W_5|\leq4\sqrt{10\varepsilon}n$, $\varepsilon^{1/3}n-\sqrt{10\varepsilon}n<|W_1|, |W_3|<(1/2-3\varepsilon_1/4)n+\sqrt{10\varepsilon}n$ and $\varepsilon^{1/3}n-\sqrt{10\varepsilon}n<|W_2|, |W_4|<(1/2-\varepsilon_1/4)n+\sqrt{10\varepsilon}n$ such that $|W_1|\approx|W_3|$ and $|W_2|\approx|W_4|$. Furthermore, the following statements hold. (See Figure \ref{fig1}$(c)$)

\smallskip

$(3.1)$ For each $\in\{1, 3\}$, apart from at most $10\sqrt{\varepsilon}|W_i|$ exceptional vertices, all vertices have the semi-degrees in $W_i$ at least $(1-10\sqrt{\varepsilon})|W_i|$, and the semi-degrees of these exceptional vertices in $W_i$ are at least $\varepsilon^{1/3}|W_i|/2$.

\smallskip

$(3.2)$ For $\{i, j\}=\{2, 4\}$, in $W_i$, apart from at most $10\sqrt{10\varepsilon}|W_i|$ exceptional vertices, all vertices in $W_i$ have bidirectional semi-degrees at least $(1-10\sqrt{\varepsilon})|W_j|$ in $W_j$, and the semi-degrees of these exceptional vertices are at least $\varepsilon^{1/3}|W_j|/8$ in $W_j$. For the vertex set $W_j$, we can get a similar conclusion by swapping all $i$ and $j$ above, respectively.

\smallskip

$(3.3)$ Without loss of generality, for any $i\in[2]$, we assume that $|W_i|\geq|W_{2+i}|$. Then for very vertex $w$ in $W_5$ and any pair $\{j_1, j_2\}\in\{\{1, 3\}, \{2, 4\}\}$, it satisfies either $d_{W_{j_1}}^+(w), d_{W_{j_2}}^-(w)\geq(1-\varepsilon^{1/3})|W_{j_1}|$ or $d_{W_{j_1}}^-(w), d_{W_{j_2}}^+(w)\geq(1-\varepsilon^{1/3})|W_{j_1}|$, and $\delta^0_{W_j}(w)\leq \varepsilon^{1/3}n/2$ for any $j\in\{1, 3\}$, and $\delta^0_{W_j}(w)\leq2\varepsilon^{1/3}n$ with $j\in\{2, 4\}$.

\smallskip

$(3.4)$ For $i\in[4]$, we have that $e^+(W_i, W_{i+1})\geq|W_i|\cdot|W_{i+1}|-\varepsilon^\prime n^2/2$, where the subscript of $W_{i+1}$ takes the remainder of modulo $4$ for $i=4$. In particular, in this case we also say that $(W_i, W_{i+1})$ is \emph{$\varepsilon$-almost one-way complete} corresponding to the vertex sets $W_i$ and $W_{i+1}$.

\bigskip

Based on the extremal condition (EC($\varepsilon$)) and the definitions of EC1($\varepsilon$), EC2($\varepsilon$) and EC3($\varepsilon, \varepsilon_1$), we can use traditional structural analysis methods to effectively demonstrate the following conclusion. Since its proof strategy essentially follows from that of Lemma $3.14$ in \cite{Wang}, it is deferred to Appendix A.
\begin{lemma}\label{claim1}
Suppose that $0<\varepsilon^\prime\ll\varepsilon_1\ll\varepsilon\ll1$. Let $D$ be an $n$-vertex digraph with $\delta^0(D)\geq n/2$. If $D$ satisfies the extremal condition \emph{(}EC$(\varepsilon^\prime)$\emph{)}, then $D$ must belong to either EC1$(\varepsilon)$, EC2$(\varepsilon)$, or EC3$(\varepsilon, \varepsilon_1)$.
\end{lemma}

Furthermore, we establish the following lemmas, which demonstrate how to properly handle the exceptional vertices in each extremal configuration of $D$ and will be frequently applied in the proofs of the main theorems. A set $\mathcal{P}$ is said to \emph{pass through} a vertex $u$ if there exists an element in $\mathcal{P}$ that contains $u$. If this holds for every vertex in a set $U$, then $\mathcal{P}$ is said to pass through $U$. We say a path $P$ is of the form $W_1W_2\cdots W_k$ if $P=w_1w_2\cdots w_k$, where $w_i\in W_i$ for each $i\in[k]$.

\begin{lemma}\label{extremal1}
Let $D$ be an $n$-vertex digraph with $\delta^0(D)\geq n/2$.
If $D$ belongs to EC1($\varepsilon$), then there is a set $\mathcal{P}$ of disjoint $W_1$-paths and $W_2$-paths passing through $W_3$. Moreover, each path in $\mathcal{P}$ has length at most $5$ and contains at most two vertices from $W_3$.
\end{lemma}
\begin{proof}
We define $W_{3, 1}$ (resp., $W_{3, 2}$) to be the set of vertices $w$ in $W_3$ satisfying that $d^-_{W_1}(w), d^+_{W_2}(w)>(1/2-\varepsilon^{1/3})n$ (resp., $d^+_{W_1}(w), d^-_{W_2}(w)>(1/2-\varepsilon^{1/3})n$).
If $|W_{3, 1}|=|W_{3, 2}|$, then for any $u\in W_{3, 1}$ and any $v\in W_{3, 2}$, by the definitions of $W_{3, 1}$ and $W_{3, 2}$, we have that
\begin{equation*}
\begin{split}
&|N^-_{W_1}(u)\cap N^+_{W_1}(v)|\geq(1-2\varepsilon^{1/3})n-|W_1|>n/3\ \mbox{and}\\
&|N^+_{W_2}(u)\cap N^-_{W_2}(v)|\geq(1-2\varepsilon^{1/3})n-|W_2|>n/3.
\end{split}
\end{equation*}
This implies that there exist many (more than $|W_{3, 1}|$) disjoint paths $P$ with the form $W_2vW_1uW_2$, and many (more than $|W_{3, 1}|$) disjoint paths $P^\prime$ of the form $W_1uW_2vW_1$. Greedily, we can get that a family $\mathcal{P}$ of disjoint paths that cover all vertices of $W_3$ and each path in $\mathcal{P}$ is a $W_1$-path or a $W_2$-path, as requested.

\smallskip

So in the following, we consider the case of $|W_{3, 1}|\neq|W_{3, 2}|$. Without loss of generality, suppose that $|W_{3, 1}|>|W_{3, 2}|$ and let $r=|W_{3, 1}|-|W_{3, 2}|$. We denote by $M_1$ (resp., $M_2$) to be a set of disjoint arcs from $W_{3, 1}$ to $W_1$ (resp., from $W_{2}$ to $W_{3, 1}$), with $M_1$ and $M_2$ being disjoint from each other. We further choose each $M_i$ ($i\in[2]$) to be maximal. Set $M=M_1\cup M_2$, and for each $i\in[2]$, define $W_{i}^\prime=W_{i}\setminus V(M)$ and $W_{3, 1}^\prime=W_{3, 1}\setminus V(M)$. We now consider two subcases.

\smallskip

\textbf{Case 1: $|M|\geq r$.} In this case, since for any $w\in W_{3, 1}$, we have
$d^-_{W_1}(w), d^+_{W_2}(w)>(1/2-\varepsilon^{1/3})n$, there are $r$ disjoint $W_1$-paths and $W_2$-paths of the forms $W_1M_1$ and $M_2W_2$. Clearly, $|W_{3, 1}|-r=|W_{3, 2}|$. Then
 for all remaining vertices $u$ in $W_{3, 1}$ and $v$ in $W_{3, 2}$, similar to the previous paragraph and by the definitions of $W_{3, 1}$ and $W_{3, 2}$, there are $|W_{3, 2}|$ disjoint $W_2$-paths (or $W_1$-paths) of form $W_2^\prime vW_1^\prime uW_2^\prime$ (or $W_1^\prime uW_2^\prime vW_1^\prime$), covering all remaining vertices in $W_{3, 1}$ and all vertices in $W_{3, 2}$, as desired.

\smallskip

\textbf{Case 2: $|M|<r$.} Here, since $M_1$ and $M_2$ are maximum and $n=|W_1^\prime|+|W_2^\prime|+|W_{3, 1}^\prime|+|W_{3, 2}|+|V(M)|$, for any vertex $u\in W_1^{\prime}$, we have that
\begin{align}\label{313}
d^-_{W_2^{\prime}}(u)&\geq\delta^0(D)-(|W_1^{\prime}|+|W_{3, 2}|+d^-_{M}(u))\nonumber\\
&\geq\frac{|W_2^\prime|-|W_1^\prime|+|W_{3, 1}^\prime|-|W_{3, 2}|+|V(M)|}{2}-d^-_{M}(u).
\end{align}
Symmetrically, for any vertex $v$ in $W_{2}^{\prime}$, we obtain that
 \begin{align}\label{314}
d^+_{W_1^{\prime}}(v)&\geq\delta^0(D)-(|W_2^{\prime}|+|W_{3, 2}|+d^+_{M}(v))\nonumber\\
&\geq\frac{|W_1^\prime|-|W_2^\prime|+|W_{3, 1}^\prime|-|W_{3, 2}|+|V(M)|}{2}-d^+_{M}(v).
\end{align}
From (\ref{313}) and (\ref{314}), for each pair of vertices $u\in W_1^{\prime}$, $v\in W_{2}^{\prime}$, we have that
\begin{align}\label{315}
|N_{W_{2}^{\prime}}^-(u)\cup N_{W_{1}^{\prime}}^+(v)|\geq|W_{3, 1}^\prime|-|W_{3, 2}|+|V(M)|-(d_M^-(u)+d_M^+(v)).
\end{align}
We first state the following conclusion:
\begin{claim}\label{123123}
For an arc $xy$ in $M$, the following holds.\\
$(i)$ If there exists a pair $(u, v)$ with $u\in W_1^\prime, v\in W_2^\prime$ such that $d^-_{xy}(u)+d^+_{xy}(v)=3$, then the value of $|W_{3, 1}^\prime|-|W_{3, 2}|$ can be reduced by $1$.\\
$(ii)$ If there exist two distinct pairs $(u_1, v_1)$, $(u_2, v_2)$ with $u_1, u_2\in W_1^\prime$, $v_1, v_2\in W_2^\prime$ $($possibly $u_1=u_2$ or $v_1=v_2$$)$ such that for each $i\in[2]$, $d^-_{xy}(u_i)+d^+_{xy}(v_i)=4$, then the value of $|W_{3, 1}^\prime|-|W_{3, 2}|$ can be reduced by $2$.
\end{claim}
\begin{proof}
(i) First, consider $xy \in M_1$, that is, $xy$ is an arc from $W_{3, 1}$ to $W_1$. If $vx\in A(D)$, then by the definition of $W_{3, 1}^\prime$, for any $w\in W_{3, 1}^\prime$, we have that $d^-_{W_1^\prime}(w), d^+_{W_2^\prime}(w)>(1-3\varepsilon^{1/3})n/2$, and for any vertex $v^\prime\in W_2^\prime$, since $D[W_2]$ is $\varepsilon$-almost complete, we have that $\delta^0_{W_2^\prime}(v^\prime)\geq(1-10\sqrt{\varepsilon})|W_2^\prime|$. This suggests that there is a vertex $v^\prime$ in $W_2^{\prime}$ satisfying $wv^\prime, v^\prime v\in A(D)$. So there is a $W_1^\prime$-path with the form $W_1^\prime wv^\prime vxy$, This path passes through one more vertex from remaining of $|W_{3, 1}^\prime|$ than the original arc $xy$ alone, effectively reducing $|W_{3, 1}^\prime|-|W_{3, 2}|$ by $1$.

\smallskip

Otherwise, $vx\notin A(D)$. Then from $d^-_{xy}(u)+d^+_{xy}(v)=3$, we have $yu, xu, vy\in A(D)$. We then replace $xy$ with $xu$ and simultaneously have a $W_1$-path of the form $W_1^\prime wv^\prime vy$, where $w\in W_{3, 1}^\prime, v^\prime\in N^+_{W_2^\prime}(w)\cap N^-_{W_2^\prime}(v)$. Note such a $v^\prime$ exists because $d^+_{W_2^\prime}(w)\geq(1-3\varepsilon^{1/3})|W_2^\prime|$ and  $\delta^0_{W_2^\prime}(v)\geq(1-10\sqrt{\varepsilon})|W_2^\prime|$. In this case, the value of $|W_{3, 1}^\prime|-|W_{3, 2}|$ is also reduced by $1$.

The case for $xy\in M_2$, that is, $x\in W_2, y\in W_{3, 1}$, is handled similarly. If $yu\in A(D)$, a $W_1$-path of the form $W_1^\prime w v^\prime xyu$ exists, where $w\in W_{3, 1}^\prime, v^\prime\in W_2^\prime$. Otherwise, we have $vx, vy, xu\in A(D)$. We replace $xy$ with $vy$ and gain an additional $W_1$-path of the form $W_1^\prime wv^\prime xyu$, where $w\in W_{3, 1}^\prime$, $v^\prime\in N^+_{W_2^\prime}(w)\cap N^-_{W_2^\prime}(x)$. In all scenarios, the value of $|W_{3, 1}^\prime|-|W_{3, 2}|$ decreases by $1$.

\smallskip

(ii) We prove only the case $xy\in M_1$; the proof for $xy\in M_2$ is analogous except for the choice of intermediate vertices. Since $d^-_{xy}(u_i)+d^+_{xy}(v_i)=4$ for each $i\in[2]$, we obtain two disjoint $W_1$-path of the forms $W_1^\prime w_1v^\prime v_2xu_1$ and $W_1^\prime w_2v^{\prime\prime}v_1y$, where $w_1, w_2\in W_{3, 1}^\prime$, $v^\prime, v^{\prime\prime}\in W_2^\prime$. Taking these two paths reduces $|W_{3, 1}^\prime|-|W_{3, 2}|$ by $2$.
\end{proof}
For an arc $xy\in M$, a pair of vertices $(u, v)\in W_1^\prime\times W_2^\prime$ is called \emph{$xy$-feasible} if either $d^-_{xy}(u)+d^+_{xy}(v)\leq3$, or $d^-_{xy}(u)+d^+_{xy}(v)=4$ and there exists another distinct pair $(u^\prime, v^\prime)\in W_1^\prime\times W_2^\prime$ such that $d^-_{xy}(u^\prime)+d^+_{xy}(v^\prime)=4$ (in which case, $(u^\prime, v^\prime)$ is called a \emph{companion pair} for $(u, v)$). Note that all companion pairs are pairwise disjoint. Furthermore, $(u, v)$ is called \emph{$M$-feasible} if it is $xy$-feasible for all $xy\in M$. We say $M$ is \emph{feasible} if all but at most $\sqrt{10\varepsilon}n^2$ pairs in $(W_1^\prime, W_2^\prime)$ is $M$-feasible.
\begin{claim}\label{feasible}
$M$ is feasible.
\end{claim}
\begin{proof}
Contrary to our claim, assume that $M$ is not feasible. Then for some arc $xy\in M$ and some pair $(u, v)\in W_1^\prime\times W_2^\prime$, we have $d^-_{xy}(u)+d^+_{xy}(v)=4$ but no companion pair for $(u, v)$. It follows that for every other pair $(u^\prime, v^\prime)\in W_1^\prime\times W_2^\prime$, except for at most $\sqrt{10\varepsilon}n$  of them, the inequality $d^-_{xy}(u^\prime)+d^+_{xy}(v^\prime)\leq3$ holds. Therefore, the number of $xy$-feasible pairs is at least $(|W_1^\prime|-\sqrt{10\varepsilon}n)\cdot(|W_2^\prime|-\sqrt{10\varepsilon}n)-\sqrt{10\varepsilon}n$. Since this argument holds uniformly for all arcs in $M$, we infer that all but at most $\sqrt{10\varepsilon}n^2$ pairs in $(W_1^\prime, W_2^\prime)$ is $M$-feasible, which verifies the feasibility of $M$.
\end{proof}
It follows from (\ref{315}) and Claim \ref{feasible} that for $i\in [2]$, there are subsets $U_i\subseteq W_i^\prime$ with $|U_i|\leq\sqrt{10\varepsilon}n$ such that all pairs $(u, v)\in W_1^\prime\setminus U_1\times W_2^\prime\setminus U_2)$ are $M$-feasible and satisfy the degree condition $|N_{W_{2}^{\prime}}^-(u)\cup N_{W_{1}^{\prime}}^+(v)|\geq|W_{3, 1}^\prime|-|W_{3, 2}|+|V(M)|-(d_M^-(u)+d_M^+(v)):=s$. Moreover, it follows from the pigeonhole principle that at least $|W_1^\prime\setminus U_1|\cdot|W_2^\prime\setminus U_2|/4$ of these pairs share the same value of $|V(M)|-(d_M^-(u)+d_M^+(v))$. By K\"{o}nig-Hall's theorem (on arcs from $W_2^\prime$ to $W_1^\prime$), this guarantees $s$ disjoint arcs $vu$ from $W_2^\prime$ to $W_1^\prime$. For $s$ distinct $w\in W_{3, 1}^\prime$, the definition of $W_{3, 1}^\prime$ gives $d^-_{W_1^\prime}(w), d^+_{W_2^\prime}(w)>(1/2-2\varepsilon^{1/3})n$. Together with the property of $W_2^\prime$, this implies $$|N_{W_2^\prime}^-(v)\cap N_{W_2^\prime}^+(w)|\geq(1-10\sqrt{\varepsilon})|W_2^\prime|+
(1/2-2\varepsilon^{1/3})n-|W_2^\prime|>n/3.$$
Thus, we obtain $s$ disjoint $W_1^\prime$-paths of the forms $W_1^\prime wW_2^\prime vu$, each utilizing a fresh arc $vu\in (W_1^\prime, W_2^\prime)$ and a new vertex $w\in W_{3, 1}^\prime$. The remaining $r-s$ vertices of $W_{3, 1}$ are passed through by disjoint paths of the forms $W_1^\prime M_1$ and $M_2W_2^\prime$ using arcs in $M$. Finally, for each $i\in[2]$, define $W_{3, i}^0$ as the remainder of $W_{3, i}$ after deleting the vertices on these paths. Clearly, $|W_{3, 1}^0|=|W_{3, 2}^0|$. This case  thereby reduces to Case $1$ discussed earlier, allowing us to find disjoint $W_1$-paths or $W_2$-paths that pass through $W_{3, 1}^0\cup W_{3, 2}^0$.

Hence, this completes the proof of Lemma \ref{extremal1}.
\end{proof}
\begin{lemma}\label{extremal2}
Let $D$ be an $n$-vertex digraph with $\delta^0(D)\geq n/2$. If $D$ belongs to EC2($\varepsilon$), then there exists a set $\mathcal{P}$ of disjoint paths that passes through the vertex set $W_3$ and satisfies $|W_1\setminus V(\mathcal{P})|=|W_2\setminus V(\mathcal{P})|$.
\end{lemma}
\begin{proof}
We use $W_{3, 1}$ (resp., $W_{3, 2}$) to define the set of vertices $w$ in $W_3$ that satisfy $d^-_{W_1}(w), d^+_{W_2}(w)>(1/2-\varepsilon^{1/3})n$ (resp., $d^+_{W_1}(w), d^-_{W_2}(w)>(1/2-\varepsilon^{1/3})n$). Clearly, we can construct two families $\mathcal{P}_1$ and $\mathcal{P}_2$ of disjoint paths with the following properties: every path in $\mathcal{P}_1$ has the alternating form $W_1W_{3, 1}W_2\cdots W_1W_{3, 1}W_2$ and $\mathcal{P}_1$ covers all vertices of $W_{3, 1}$; similarly, every path in $\mathcal{P}_2$ follows the form $W_2W_{3, 2}W_1\cdots W_2W_{3, 2}W_1$, and $\mathcal{P}_2$ covers all vertices of $W_{3, 2}$.

\smallskip

In the following, we may assume $|W_1|\neq|W_2|$. Otherwise, taking $\mathcal{P}=\{\mathcal{P}_1, \mathcal{P}_2\}$ immediately proves the lemma. Without loss of generality, suppose $|W_1|>|W_2|$ and set $r=|W_1|-|W_2|$. We further define the set of disjoint arcs from $W_{3, 1}$ to $W_1$ (resp., from $W_1$ to $W_{3, 2}$) as $M_1$ (resp., $M_2$), and $M_1$ and $M_2$ are also disjoint. Choose each $M_i$ ($i\in[2]$) to be maximal, and set $M=M_1\cup M_2$. If $|M|=r$, then, according to the properties of $W_{3, 1}$ and $W_{3, 2}$, we can find $r$ disjoint $3$-paths of the forms $W_1M_1$ and $M_2W_1$. These $3$-paths can be directly incorporated into the construction of $\mathcal{P}_1$ and $\mathcal{P}_2$, described above. Hence, the lemma holds with $\mathcal{P}=\{\mathcal{P}_1, \mathcal{P}_2\}$.

\smallskip

Now suppose that $|M|<r$. Define $W_i^\prime=W_i\setminus V(M)$ for $i\in\{1, 3\}$. Combining the semi-degree condition $\delta^0(D)\geq n/2$ and $n=|W_1^\prime|+|V(M)|+|W_2|+|W_3^\prime|$, we obtain that for any vertex $u\in W_1^\prime$,
\begin{align}\label{22m}
d_{W_1^\prime}(u)&\geq2\delta^0(D)-(|W_3^\prime|+2|W_2|+d_M(u))\nonumber\\
&\geq|W_1^\prime|-|W_2|+|V(M)|-d_M(u).
\end{align}

\begin{claim}\label{123123123}
For an arc $xy\in M$, the following holds.\\
$(i)$ If there exists a vertex $u\in W_1^\prime$ with $d_{xy}(u)=3$, then the value of $|W_1^\prime|$ can be effectively reduced by $1$.\\
$(ii)$ If there exists a pair of vertices $u_1, u_2\in W_1^\prime$ with $d_{xy}(u_i)=4$ for each $i\in[2]$, then the value of $|W_1^\prime|$ can be effectively reduced by $2$.
\end{claim}
\begin{proof}
We only give the proof for (i) and (ii) when $xy \in M_1$, the case $xy \in M_2$ is analogous.

(i) If\ $yu\in A(D)$, replacing $xy$ with the path $xyu$ reduces the value of $|W_1^\prime|$ by $1$. Otherwise, $d_{xy}(u)=3$ suggests that $uy, ux, xu\in A(D)$. In this case, replace $xy$ with $xuy$ to achieve the same effect.

(ii) Since $d_{xy}(u_i)=4$ for each $i\in[2]$, replacing $xy$ with the path $xu_1yu_2$ reduces the value of $|W_1^\prime|$ by $2$.
\end{proof}
Similar to the proof of Lemma \ref{extremal1}, for an arc $xy\in M$, we call a vertex $u\in W_1^\prime$  \emph{$xy$-feasible} if either $d_{xy}(u)\leq3$, or $d_{xy}(u)=4$ and there exists a distinct vertex $u^\prime W_1^\prime$ (called a \emph{companion} of $u$) such that $d_{xy}(u^\prime)=4$. Note that all companions are required to be distinct from each other. We also call $u\in W_1^\prime$ to be \emph{$M$-feasible} if it is $xy$-feasible for any $xy\in M$. Furthermore, we say $M$ is \emph{feasible} if all but at most $2\sqrt{10\varepsilon}n$ vertices in $W_1^\prime$ is $M$-feasible.
\begin{claim}\label{ffeasible}
$M$ is feasible.
\end{claim}
\begin{proof}
 If for every vertex $u\in W_1^\prime$ and every arc $xy\in M$,we have $d_{xy}(u)\leq3$, or $d_{xy}(u)=4$ with approximately $2\sqrt{10\varepsilon}n$ available companions of $u$, then $M$ is clearly feasible. Hence, we assume that there exists a vertex $u_1\in W_1^\prime$ and an arc $x_1y_1\in M$ such that $d_{x_1y_1}(u_1)=4$, but $u_1$ has no available companion (note that $u_1$ may have had companions, but if those partners are already assigned to other vertices, they cannot be reused for $u_1$). This implies that for all but at most $2\sqrt{10\varepsilon}n$ such vertices $u\in W_1^\prime$, we have $d_{x_1y_1}(u)\leq3$. Similarly, considering all arcs $xy\in M$, we get that except for at most $2\sqrt{10\varepsilon}n$ vertices, every remaining $u\in W_1^\prime$ satisfies $d_M(u)\leq3|M|$. This yields that $M$ is feasible.
\end{proof}

 Recall from (\ref{22m}) that $d_{W_1^\prime}(u)\geq|W_1^\prime|-|W_2|+|V(M)|-d_M(u):=s$. Applying the pigeonhole principle, at least $(|W_1^\prime|-2\sqrt{10\varepsilon}n)/4$ of these vertices $u\in W_1^\prime$ share the same value of $|V(M)|-d_M(u)$. This yields an additional set $M_3$ of $s$ disjoint arcs in $D[W_1^\prime]$. Since $(W_1, W_2)$ is $\varepsilon$-almost bipartite complete, Claim \ref{ffeasible} allows us to combine the arcs in $M\cup M_3$ with some vertices from $W_1^\prime\cup W_2\cup W_3^\prime$ to obtain a set $\mathcal{P}_3$ of disjoint $(W_1, W_2)$-paths such that $|V(W_1)\setminus V(\mathcal{P}_3)|=|V(W_2)\setminus V(\mathcal{P}_3)|$. For the remaining vertices in $W_3$, we can, as we did at the beginning of the proof, construct two disjoint sets of $(W_1, W_2)$-paths $\mathcal{P}_1$ and $\mathcal{P}_2$, that together cover all of these remaining vertices of $W_3$. Finally, let $\mathcal{P}=\{\mathcal{P}_1, \mathcal{P}_2, \mathcal{P}_3\}$ as requested. Hence this completes the proof of Lemma \ref{extremal2}.
\end{proof}
\begin{lemma}\label{extremal3}
Suppose that $0<\varepsilon^\prime\ll\varepsilon_1\ll\varepsilon\ll1$. Let $D$ be an $n$-vertex digraph with $\delta^0(D)\geq n/2$. If $D$ belongs to EC3($\varepsilon, \varepsilon_1$), then there is set $\mathcal{P}$ of disjoint $W_1$-paths, $W_3$-paths, $(W_2, W_4)$-paths and $(W_4, W_2)$-paths such that $|W_2\setminus V(\mathcal{P})|=|W_4\setminus V(\mathcal{P})|$, and $\mathcal{P}$ passes through the vertex set $W_5$.
\end{lemma}
\begin{proof}
We prove this lemma by the following two claims, the first of which states that there is a set $\mathcal{P}^\prime$ of disjoint paths with length at most $4$, such that for any $w\in W_5$ satisfying that $d^\sigma_{W_1}(w), d^{-\sigma}_{W_3}(w)\geq(1-\varepsilon^{1/3})|W_1|$ for some $\sigma\in\{+, -\}$, we have $\mathcal{P}^\prime$ through $w$; and the second of which shows that there exists a set $\mathcal{P}^{\prime\prime}$ of disjoint paths such that it passes through any $w\in W_5$ satisfying that $d^\sigma_{W_2}(w), d^{-\sigma}_{W_4}(w)\geq(1-\varepsilon^{1/3})|W_1|$. Thus, taking  $\mathcal{P}=\{\mathcal{P}^\prime, \mathcal{P}^{\prime\prime}\}$ completes the proof of the lemma.
\begin{claim}\label{321321}
For any vertex $w\in W_5$, if for some $\sigma\in\{+, -\}$ we have $d_{W_1}^\sigma(w), d_{W_3}^{-\sigma}(w)\geq(1-\varepsilon^{1/3})|W_1|$, then for any $i\in\{1, 3\}$, there is a $W_i$-path passing through $w$.
\end{claim}
\begin{proof}
 First, consider any vertex $w\in W_5$ such that $d^-_{W_1}(w), d^+_{W_3}(w)>(1-2\varepsilon^{1/3})|W_1|$. Since $(W_3, W_4)$ is $\varepsilon$-almost one-way complete, we get that $e^+(N^+_{W_3}(w), W_4)\geq|N^+_{W_3}(w)|\cdot|W_4|-\varepsilon^\prime n^2/2$. Consequently, there exist at least $|N^+_{W_3}(w)|-\varepsilon^\prime n\geq(1-3\varepsilon^{1/3})|W_1|$ vertices in $W_3$ with the following property: for each such vertex $u$, we have $d^+_{W_4}(u)\geq|W_4|/2\geq\varepsilon^{1/3}n/4$. Next, using the fact that $(W_4, W_1)$ is also $\varepsilon$-almost one-way complete, we obtain $e^+(N^+_{W_4}(u), W_1)\geq|N^+_{W_4}(u)|\cdot|W_1|-\varepsilon^\prime n^2/2$. The above inequalities imply the existence of a path of the form $wW_3W_4W_1$. On the other hand, since $d^-_{W_1}(w)>(1-2\varepsilon^{1/3})|W_1|$, there is an arc from some vertex in $W_1$ to $w$. Combining these, we obtain a $W_1$-path of length at most $4$ that passes through $w$.

\smallskip

By symmetry, we can also find a $W_3$-path of length at most $4$ passing through $w$. Indeed, from the $\varepsilon$-almost one-way completeness of $(W_4, W_1)$, we have $e^-(W_1, W_4)\geq|W_1|\cdot|W_4|-\varepsilon^\prime n^2/2$. Together with $|N^-_{W_1}(w)|\geq(1-\varepsilon^{1/3})|W_1|$, we get that there exists a vertex set $U\subset W_4$ with  $|U|\geq|W_4|-\varepsilon^\prime n$ such that for any vertex $u\in U$, we have $|N^+(u)\cap N^-_{W_1}(w)|\geq|N^-_{W_1}(w)|/2$. Further, since $(W_3, W_4)$ is $\varepsilon$-almost one-way complete, we also have $e^-(U, W_3)\geq |U|\cdot|W_3|-\varepsilon^\prime n^2/2$. These facts guarantee a path of the form $W_3W_4W_1w$. Together with the arc from $w$ to $W_3$ (which follows from  $d_{W_3}^+(w)\geq(1-2\varepsilon^{1/3})|W_1|$), we obtain the desired $W_3$-path passing through $w$.

\smallskip

Finally, for any vertex $w$ satisfying $d^+_{W_1}(w), d^-_{W_3}(w)\geq(1-2\varepsilon^{1/3})|W_1|$, we can similarly prove that there is a $W_1$-path ($W_3$-path) of length at most $4$ through it. Since the proof strategy of this case essentially follows from that in the preceding two paragraphs, we omit the repetitive details. This completes the proof of the claim.
\end{proof}
We define $W_{5, 1}$ (resp., $W_{5, 3}$) as the set of vertices $w\in W_5$ satisfying $d^-_{W_1}(w), d^+_{W_3}(w)>(1/2-\varepsilon^{1/3})|W_1|$ (resp., $d^+_{W_1}(w), d^-_{W_3}(w)>(1/2-\varepsilon^{1/3})|W_1|$). If $|W_{5, 1}|=|W_{5, 3}|$, then as established in the proof of Lemma \ref{extremal1}, we can find a set $\mathcal{P}^\prime$ of disjoint $W_1$-paths with the form $W_1W_{5, 1}W_3W_{5, 3}W_1$, and passing through $W_{5, 1}\cup W_{5, 3}$. So in the following, we consider the case of $|W_{5, 1}|\neq|W_{5, 3}|$. Without loss of generality, suppose that $|W_{5, 1}|>|W_{5, 3}|$. By repeatedly applying Claim \ref{321321} to $|W_{5, 1}|-|W_{5, 3}|$ vertices $w\in W_{5, 1}$, we obtain  $|W_{5, 1}|-|W_{5, 3}|$ disjoint $W_1$-paths of the form $W_1W_2W_3wW_1$. This can be done because the number of arcs between the relevant vertex sets significantly exceeds the cardinality of $W_5$. During this process, we use at most $10\varepsilon^{1/2}|W_2|$ from $W_2$ and at most $10\varepsilon^{1/2}|W_4|$ vertices from $W_4$. For the remaining in $W_{5, 3}$ and all vertices in $W_{5, 1}$, we can, as before, find disjoint $W_1$-paths of the form $W_1W_{5, 1}W_3W_{5, 3}W_1$. Let $\mathcal{P}^\prime$ be the set of all these disjoint $W_1$-paths. By construction, $\mathcal{P}^\prime$ passes through $W_{5, 1}\cup W_{5, 3}$.
For convenience, set $W_i^\prime=W_i\setminus V(\mathcal{P}^\prime)$ for each $i\in[5]$. Clearly, $|W_i^\prime|\geq\varepsilon^{1/3}n/4$. Moreover, for any $w\in W_5^\prime$, we have either $d^-_{W_2}(w), d^+_{W_4}(w)\geq(1-\varepsilon^{1/3})|W_2|$ or $d^-_{W_4}(w), d^+_{W_2}(w)\geq(1-\varepsilon^{1/3})|W_2|$.
\begin{claim}\label{231213}
There is a set $\mathcal{P}^{\prime\prime}$ of disjoint $W_1$-paths, $W_3$-paths, $(W_2, W_4)$-paths and $(W_4, W_2)$-paths, satisfying that $|W_2^\prime\setminus V(\mathcal{P}^{\prime\prime})|=|W_4^\prime\setminus V(\mathcal{P}^{\prime\prime})|$, and $\mathcal{P}^{\prime\prime}$ passes through $W_5^\prime$.
\end{claim}
\begin{proof}
We denote by $W_{5, 1}^\prime$ (resp., $W_{5, 2}^\prime$) to be the set of vertices $w\in W_5^\prime$ satisfying that $d^-_{W_2}(w), d^+_{W_4}(w)\geq(1-\varepsilon^{1/3})|W_2|$ (resp., $d^-_{W_4}(w), d^+_{W_2}(w)\geq(1-\varepsilon^{1/3})|W_2|$). Clearly, since $(W_2, W_4)$ is an $\varepsilon$-almost complete bipartite digraph, we can construct two families of disjoint paths $\mathcal{P}_1$ and $\mathcal{P}_2$ with the following properties: each path in $\mathcal{P}_1$ has the form $W_2^\prime W_{5, 1}^\prime W_4^\prime\cdots W_2^\prime W_{5, 1}^\prime W_4^\prime$, and $\mathcal{P}_1$ passes through $W_{5, 1}^\prime$; and each path in $\mathcal{P}_2$ has the form $W_4^\prime W_{5, 2}^\prime W_2^\prime \cdots W_4^\prime W_{5, 2}^\prime W_2^\prime$, and $\mathcal{P}_2$ passes through $W_{5, 2}^\prime$. If $|W_2^\prime|=|W_4^\prime|$, then we take $\mathcal{P}^{\prime\prime}=\{\mathcal{P}_1, \mathcal{P}_2\}$ as required.

\smallskip

Therefore, in what follows we assume $|W_2^\prime|\neq|W_4^\prime|$. First we consider the case when $|W_2|\geq|W_4|$ and $|W_2^\prime|\geq|W_4^\prime|$. We take a family $\mathcal{P}_0$ (maybe, $\mathcal{P}_0=\mathcal{P}^\prime$ and not necessarily disjoint from $\mathcal{P}^\prime$) of disjoint $W_1$-paths, $W_3$-paths, $(W_2, W_4)$-paths and $(W_4, W_2)$-paths such that the positive difference $r:=|W_2^\prime\setminus V(\mathcal{P}_0)|-|W_4^\prime\setminus V(\mathcal{P}_0)|$ is minimized. Let $W_2^{\prime\prime}:=W_2^\prime\setminus V(\mathcal{P}_0)$, and then take an arbitrary vertex $u\in W_2^{\prime\prime}$. Then firstly, we have $d^+_{W_1}(u)=0$ and $d^-_{W_3}(u)=0$. If not, we could extend $\mathcal{P}_0$ by a suitable $W_1$-path and a $W_3$-path, respectively, containing $u$, which would reduce $r$ by $1$, contradicting the minimality of $\mathcal{P}_0$. Secondly, for any vertex $w\in W_5^\prime$ satisfying $d_{W_2}^-(w), d_{W_4}^+(w)\geq(1-\varepsilon^{1/3})|W_2|$ (resp., $d_{W_2}^+(w), d_{W_4}^-(w)\geq(1-\varepsilon^{1/3})|W_2|$), the minimality of $\mathcal{P}_0$ implies $wu\notin A(D)$ (resp., $uw\notin A(D)$). Otherwise, a path adjustment similar to the proof of Lemma \ref{extremal2} would reduce $r$ by $1$. Thirdly, for each path $P$ of the form $W_1w_1W_3w_2W_1$ where $w_1\in W_{5, 1}$ and $w_2\in W_{5, 3}$, we have $d_{\{w_1\}}(u)=0$ and $d_{\{w_2\}}(u)\leq2$. Since, otherwise, we can replace $P$ with the path of the form $W_1uw_1W_3w_2W_1$ and $W_1w_1uW_3w_2W_1$, respectively, a contraction with the minimality of $r$. Fourthly, each path $P$ of the form $W_1u^\prime W_3wW_1$ where $u^\prime\in W_2$ and $w\in W_{5, 3}$, we have $d_{\{u^\prime\}}(u)=0$ and $d_{\{w\}}(u)\leq2$. Otherwise, we can replace $P$ with the path of the form $W_1W_2W_2W_3wW_1$, a contraction with the minimality of $r$ again.

Consequently, from the four properties above, we can conclude that $d_{\mathcal{P}_0}(u)\leq|V(\mathcal{P}_0)|$. Together with the lower bound of $\delta^0(D)$, we have $d_{W_2^{\prime\prime}}(u)\geq|W_2^{\prime\prime}|-|W_4^{\prime\prime}|=r$. This implies that there exist $r$ pairwise disjoint arcs in $D[W_2^{\prime\prime}]$. Recall that for any vertex $w\in W_5^\prime$, it satisfies $d^\sigma_{W_{2}^\prime}(w), d^{-\sigma}_{W_{4}^\prime}(w)>(1-2\varepsilon^{1/3})|W_{2}^\prime|$ for some $\sigma\in\{+, -\}$. Since $(W_2^{\prime\prime}, W_4^{\prime\prime})$ is an $\varepsilon$-almost complete bipartite digraph, similar to the proof of Lemma \ref{extremal2}, there exist disjoint $(W_2, W_4)$-paths that cover these $r$ arcs and all vertices in $W_5^\prime$. Let $\mathcal{P}^{\prime\prime}$ be the union of these paths and $\mathcal{P}_0$, then we have $|W_2^{\prime\prime}\setminus V(\mathcal{P}^{\prime\prime})|=|W_4^{\prime\prime}\setminus V(\mathcal{P}^{\prime\prime})|$.

The proof for the remaining cases is analogous to that of the case presented above. Therefore, we omit it to avoid redundancy.
\end{proof}
Hence by the two claims above, this completes the proof of this lemma.
\end{proof}
Before giving the proofs of Theorems \ref{song1} and \ref{song2}, we also need the following conclusions that can help us to iteratively find disjoint $H$-subdivisions of arbitrary orders in a dense digraph with sufficiently large semi-degree.
\begin{proposition}\label{we}
Suppose that $C$ is a positive integer and $\eta$ is any real number with $0<1/C\ll\eta\ll1$. For any integer partition $a=a_1+\cdots+a_k$ with $a\geq Ck$, let $T=(A, B)$ be a balanced bipartite digraph of order $2a$. For any $\sigma\in\{+, -\}$, if $T$ satisfies that for any vertex $u\in A$ and any $v\in B$, $d^\sigma_B(u)\geq(1-\eta)a$ and  $d^\sigma_A(v)\geq(1-\eta)a$, then for any vertex set $U\subseteq V(T)$ with $U\cap A=\{x_1^0, \ldots, x_k^0\}$ and $U\cap B=\{y_1^0, \ldots, y_k^0\}$, $T$ contains $k$ disjoint paths $P_1, \ldots, P_k$ such that for each $i\in[k]$, the initial and terminal of $P_i$ is $x_i^0$ and $y_i^0$, respectively, and $|V(P_i)\cap A|=|V(P_i)\cap B|=a_i$.
\end{proposition}
\begin{proof}
For convenience, for any $i\in[k]$, let $r_i=a_i-1$.
For each $i\in[k]$, we choose $r_i+1$ vertices $x_i^0, x_i^1, \ldots, x_i^{r_i}$ from $A$ with the last vertex $x_i^{r_i} \in N^-_{A}(y_i^0)$ such that all these vertices are distinct and their union is $A$.  We construct an auxiliary bipartite graph $Q=(\tilde{A}, B^\prime)$ such that $\tilde{A}=\bigcup_{i=1}^k\{(x_i^0, x_i^1), (x_i^1, x_i^2), \ldots, (x_i^{r_i-1}, x_i^{r_i})\}$ and $B^\prime=B\setminus U$, where each `vertex' $(x_i^j, x_i^{j+1})$ in $\tilde{A}$ connects with all the vertices in $N_{T}^+(x_i^j)\cap N_{T}^-(x_i^{j+1})$. Obviously, any perfect matching in $Q$ that saturates $\tilde{A}$ corresponds to an embedding of disjoints $P_1, \ldots, P_k$ in $T$ as required. We claim that such perfect matching exists. In fact, $|\tilde{A^\prime}|=|B^{\prime}|=a-k$ and $d_Q(z)\geq2(1-\eta)|B^\prime|-|B^\prime|\geq(1-2\eta)|B^\prime|$ for $z\in \tilde{A}$. Additionally, we deduce that $d_Q(u)\geq(1-2\eta)|\tilde{A}|$ for any vertex $u\in B^\prime$. Therefore, the degrees of the vertices in $Q$ are all at least $(1-2\eta)(a-k)$. By the K\"{o}nig-Hall's theorem, we conclude that $Q$ contains a perfect matching, which completes the proof this proposition.
\end{proof}
Recall that for a digraph $H$, we have $m_H=\min_{H=(U_1, U_2)}\big||U_1|+e(U_2)-(|U_2|+e(U_1))\big|$. For a bipartition $(U_1, U_2)$ attaining this minimum, we further set $m_H^\prime=|V(H)|+e(U_2)+e(U_1).$
Using the proposition above, we obtain the following result.
\begin{lemma}\label{qaz}
Let $H$ be any digraph with $h$ arcs and $\delta(H)\geq1$. Assume that $0<1/C\ll\xi\ll1$ and let $a=a_1+\cdots+a_k$ satisfy $a\geq Chk$ and $a_i\geq(3h+|V(H)|-m_H^\prime)/2$ for each $i$. Let $T$ be a digraph with bipartition $V(T)=A\cup B$, where $|A|-|B|=t$ with $t=0$ if $k\equiv0\ (\mbox{mod}\ 2)$ and $t=m_H$ if $k\equiv1\ (\mbox{mod}\ 2)$. If for any $\sigma\in\{+, -\}$, every $u\in A$ and $v\in B$, we have that $d_B^\sigma(u), d_A^\sigma(v)\geq(1-\xi)a$, then the following hold.\\
$(i)$ $T$ contains a spanning $H$-subdivision.\\
$(ii)$ $T$ admits a perfect $H$-subdivision tiling in which the $i$-th subdivision has order $m_H^\prime+2n_i$, for $i\in[k]$.
\end{lemma}
\begin{proof}
Without loss of generality, assume that the bipartition of $H$ achieving $m_H$ is $(U_1, U_2)$ with $m_H=|U_1|+e(U_2)-(|U_2|+e(U_1))>0$. Then, by definition, $m_H^\prime=|U_1|+e(U_2)+(|U_2|+e(U_1))$. Write $U_1=\{a_1, \ldots, a_r\}, U_2=\{b_1, \ldots, b_s\}$. We first prove statement (i), and then statement (ii).

(i) In this case, we have $k=1$. Consequently, $|A|=|B|+m_H$. We select the image set of $U_1$ from $A$ and the image set of $U_2$ from $B$. For simplicity, we still denote each image set by $U_i$ for $i\in[2]$, with $U_1=\{a_1, \ldots, a_r\}, U_2=\{b_1, \ldots, b_s\}$. Below we first use steps (i1) and (i2) to associate each arc in $H$ with two short subpaths; we then apply Proposition \ref{we} to concatenate these subpaths pairwise into one long path.

\smallskip

\textbf{Step (i1): arcs between $U_1$ and $U_2$.} Consider a pair $(a_i, b_j)\in U_1\times U_2$. If $a_ib_j\in A(H)$, then, using the lower bounds on the semi-degrees of vertices in $A$ and $B$, we can find two disjoint paths $P_{i, j}^1=a_iv_{i, j}a_{i, j}^1$ and $Q_{i, j}^1=b_ju_{i, j}b_{i, j}^1$ with $a_{i, j}^1, u_{i, j}\in A\setminus U_1$ and $v_{i, j}, b_{i, j}^1\in B\setminus U_2$. If $b_ja_i\in A(H)$, then similarly, we find two disjoint paths $P_{i, j}^2=a_ib_{i, j}^2$ and $Q_{i, j}^2=b_ja_{i, j}^2$ with $a_{i, j}^2\in A\setminus U_1$ and $b_{i, j}^2\in B\setminus U_2$.

\smallskip

\textbf{Step (i2): arcs inside $U_1$ and $U_2$.} Consider a pair $(a_i, a_j)\in U_1\times U_1$. If $a_ia_j\in A(H)$, then using the abundance of arcs between $A$ and $B$ again, we obtain disjoint paths $R_{i, j}^1=a_iv_{i, j}^2a_{i, j}^3$ and $R_{i, j}^2=b_{i, j}^3a_j$, where $v_{i, j}^2, b_{i, j}^3\in B\setminus U_2$ and $a_{i, j}^3\in A\setminus U_1$. Similarly, consider a pair $(b_i, b_j)\in U_2\times U_2$. We find two disjoint paths $S_{i, j}^1=b_ia_{i, j}^4$ and $S_{i, j}^2=b_{i, j}^4u_{i, j}^3b_j$ with $a_{i, j}^4, u_{i, j}^3\in A\setminus U_1$ and $b_{i, j}^4\in B\setminus U_2$.

\smallskip

Note that all vertices outside $U_1\cup U_2$ that appear in these paths are chosen distinctly. At each choice we avoid previously used vertices. Because the total number of new vertices employed is at most $4h$ and every vertex still has at least $(1-\xi)a-4h\geq(1-3\xi/2)a$ (since $h\leq a/(Ck)$ and $1/C\ll\xi$) available neighbours, the greedy selection succeeds.

\smallskip

After performing the corresponding step for every arc of $H$, we obtain a family of internally disjoint paths that realise the connections prescribed by $H$. The total number of vertices consumed is $2e^+(U_1, U_2)+e^-(U_1, U_2)+e(U_1)+2e(U_2)$ from $A$ and $2e^+(U_1, U_2)+e^-(U_1, U_2)+2e(U_1)+e(U_2)$ from $B$. Let $A_0$ (respectively, $B_0$) denote the subset of $A$ (respectively, $B$) obtained by deleting all vertices that lie on any of the paths constructed above. Clearly, since $|U_1|+e(U_2)-(|U_2|+e(U_1))=m_H$, we have $|A_0|=|B_0|$.

Recall that $r=|U_1|, s=|U_2|$. Sequentially, for every $\sigma\in\{+, -\}$, every $u\in A_0\cup\bigcup_{i\in[r], j\in[s], l\in[4]}\{a_{i, j}^l\}$ and $v\in B_0\cup\bigcup_{i\in[r], j\in[s], l\in[4]}\{b_{i, j}^l\}$ (where we allow that some sets $\{a_{i_0, j_0}\}$ and $\{b_{i_1, j_2}\}$ may be empty for certain indices $i_0, i_1\in[r], j_0, j_1\
in[s]$), we have that $d_{B^\prime}^\sigma(u), d_{A^\prime}^\sigma(v)\geq(1-\xi)a-4h-|V(H)|$. Using the facts that $|V(H)|\leq2h$, $h\leq\frac{a}{Ck}$ and $0<1/C\ll\xi\ll1$, we obtain
\begin{align*}
d_{B^\prime}^\sigma(u), d_{A^\prime}^\sigma(v)\geq(1-\xi)a-4h-2h\geq(1-\xi-6/(Ck))a\geq(1-\xi-6\xi/k)a\geq(1-2\xi)a.
\end{align*}

To complete a spanning $H$-subdivision, it now suffices to find, in almost complete bipartite digraph $(A_0, B_0)$, a collection of disjoint paths $P_1, \ldots, P_{h}$, such that for each $i$, the initial of $P_i$ is some $a_{i, j}^l$ and its terminal vertex is the corresponding $b_{i, j}^l$ (where $l\in[4]$). This is guaranteed by Proposition \ref{we} with $2\xi$ playing the role of $\eta$. Hence we have shown that $T$ contains a spanning $H$-subdivision, proving the statement (i).

\smallskip

(ii) We select $\lceil k/2\rceil$ image sets of $U_1$ and $\lfloor k/2\rfloor$ image sets of $U_2$ from $A$, and correspondingly $\lceil k/2\rceil$ image sets of $U_2$ and $\lfloor k/2\rfloor$ image sets of $U_1$ from $B$. This yields a total of $k$ matched pairs. Label the first $\lceil k/2\rceil$ pairs as $(U_{1, 1}, U_{2, 1}),\ldots, (U_{1, \lceil k/2\rceil}, U_{2, \lceil k/2\rceil})$, and the remaining $\lfloor k/2\rfloor$ pairs as $(U_{2, \lceil k/2\rceil+1}, U_{1, \lceil k/2\rceil+1}), \ldots, (U_{2, k}, U_{1, k})$.

\smallskip

For each pair $(U_{1, i}, U_{2, i})$ with $i\in[\lceil k/2\rceil]$ and each pair $(U_{2, j}, U_{1, j})$ with $j\in\{\lceil k/2\rceil+1, \ldots, k\}$, we carry out steps (i1) and (i2) from part (i) analogously. This process constructs, for each index $i$, a set of $2h$ paths $P^1_{i, 1}, P^2_{i, 1}\ldots, P^1_{i, h}, P^2_{i, h}$ whose whose intersection is contained in $U_{1, i}\cup U_{2, i}$, and for each index $j$, a set of $2h$ paths $Q^1_{j, 1}, Q^2_{j, 1}\ldots, Q^1_{j, h}, Q^2_{j, h}$ whose intersection lies in $U_{2, j}\cup U_{1, j}$. These constructions satisfy the balance properties that $|A\cap\bigcup_{l\in[h]}V(P^1_{i, l}\cup P^2_{i, l})|-|B\cap\bigcup_{l\in[h]}V(P^1_{i, l}\cup P^2_{i, l})|=m_H$ and $|A\cap\bigcup_{l\in[h]}V(Q^1_{j, l}\cup Q^2_{j, l})|-|B\cap\bigcup_{l\in[h]}V(Q^1_{j, l}\cup Q^2_{j, l})|=-m_H$.

\smallskip

Let $A_0$ and $B_0$ denote the subsets of $A$ and $B$, respectively, obtained by deleting all vertices appearing in any of the paths constructed above. Since each construction for index $i\in[k]$ contributes a net imbalance of either $m_H$ or $-m_H$, summing over all $k$ indices and taking into account the definition of $t$, we obtain $|A_0|=|B_0|$. Moreover, for any $\sigma\in\{+, -\}$, every vertex $u\in A_0$ and $v\in B_0$ satisfies
$$d_{B_0}^\sigma(u), d_{A_0}^\sigma(v)\geq(1-\xi)a-\xi a/10-6hk-2k\geq(1-\xi)a-7a/C\geq(1-2\xi)a.$$
This follows because $|V(H)|\leq2h$, $h\leq a/(Ck)$ and the inequality $0<1/C\ll\xi\ll1$. 

\smallskip

Finally, we will apply Proposition \ref{we} with $2\xi$ playing the role of $\eta$. This yields a collection of disjoint $(A_0, B_0)$-paths such that: \textbf{(ii1)} for each $i\in[\lceil k/2\rceil]$ and $l\in[h]$, there is a path of specified length from the terminal of $P_{i, l}^1$ to the initial of $P_{i, l}^2$; \textbf{(ii2)} for $j\in\{\lceil k/2\rceil+1, \ldots, k\}$ and $l\in[h]$, there is a path of specified length from the terminal of $Q_{j, l}^1$ to the initial of $Q_{j, l}^2$.


This completes the construction of a perfect $H$-subdivision tiling in $T$ where the subdivisions have orders $m_H^\prime+2a_1, \ldots, m_H^\prime+2a_k$, respectively, thereby proving (ii).
\end{proof}

\subsection{Proof of Theorem \ref{song1}}
Let $C\geq C_0$ be an integer and choose parameters $\varepsilon^\prime, \varepsilon_1$ and $\varepsilon$ satisfying the hierarchy $1/C\ll\varepsilon^\prime\ll\varepsilon_1\ll\varepsilon\ll1$. In this section, we present the proof of Theorem \ref{song1} for the case where $D$ is not $\varepsilon^\prime$-stable.
\begin{proof}
Let $H$ be a digraph with $h$ arcs and minimum degree $\delta(H)\geq1$, and let $D$ be a digraph on $n\geq Ch$ vertices with $\delta^0(D)\geq\frac{n+h}{2}-1$, as described in Theorem \ref{song1}. If $h=1$, which corresponds to the case where $H$ consists of a single arc, then a spanning $H$-subdivision in $D$ is exactly a Hamiltonian path. In this scenario, the minimum semi-degree satisfies $\delta^0(D)\geq\frac{n+h}{2}-1=\frac{n-1}{2}$. By Ghouila-Houri's theorem \cite{Ghouila-Houri}, this minimum semi-degree condition ensures that $D$ contains a Hamiltonian path. Therefore, in what follows, we assume $h\geq2$. By Lemma \ref{claim1}, the digraph $D$ belongs to EC1($\varepsilon$), EC2($\varepsilon$), or EC3($\varepsilon, \varepsilon_1$). 

\smallskip

\textbf{Case a.1: $D$ belongs to EC1($\varepsilon$).} In this case, the vertex set $V(D)$ can be partitioned into three disjoint sets $W_1, W_2$ and $W_3$, such that for each $i\in[2]$, we have $|W_i|=(1/2-\varepsilon)n\pm\sqrt{10\varepsilon}|W_i|$, and each induced subdigraph $D[W_i]$ is $\varepsilon$-almost complete. Since $\delta^0(D)\geq (n+h)/2-1$, for any vertex $u\in W_1$ and any vertex $v\in W_2$, the following holds.
\begin{align*}
\begin{split}
\left \{
\begin{array}{ll}
\emph{\emph{$uv\in A(D)$ and $|N^+(u)\cap N^-(v)|\geq2\delta^0(D)-(n-1)\geq h-1$, or}}\\
\emph{\emph{$uv\notin A(D)$ and $|N^+(u)\cap N^-(v)|\geq 2\delta^0(D)-(n-2)\geq h$}}.
\end{array}
\right.
\end{split}
\end{align*}
Symmetrically, we also have that either
\begin{align*}
\begin{split}
\left \{
\begin{array}{ll}
\emph{\emph{$vu\in A(D)$ and $|N^-(u)\cap N^+(v)|\geq2\delta^0(D)-(n-1)\geq h-1$, or}}\\
\emph{\emph{$vu\notin A(D)$ and $|N^-(u)\cap N^+(v)|\geq 2\delta^0(D)-(n-2)\geq h$}}.
\end{array}
\right.
\end{split}
\end{align*}
This implies the existence of two disjoint minimal paths $P_1=u_1\cdots v_1$ and $P_2=v_1^\prime\cdots u_1^\prime$, each of length at most $2$, such that $P_1$ is a $(W_1, W_2)$-path and $P_2$ is a $(W_2, W_1)$-path. 
Let $W_3^\prime=W_3\setminus V(P_1\cup P_2)$.
Furthermore, by Lemma \ref{extremal1}, there exist two disjoint path sets $\mathcal{P}_1$ and $\mathcal{P}_2$, such that $\mathcal{P}_1\cup\mathcal{P}_2$ covers all vertices of $W_3^\prime$, and for each $j\in[2]$, $\mathcal{P}_j$ is a set of disjoint $W_i^\prime$-paths of length at most $4$. Note that it is possible that $P_1$ and $P_2$ are subpaths of some path(s) in $\mathcal{P}_1\cup\mathcal{P}_2$.
Utilizing property $(1.1)$ of EC1($\varepsilon$) and Proposition \ref{we}, after first dealing with a bounded number of exceptional vertices with not too big semi-degrees in $W_2$, we obtain a Hamiltonian path $P$ in $D[W_2\cup V(\mathcal{P}_2)]$ that connects $v_1$ to $v_1^\prime$.

\smallskip

Furthermore, the property $(1.1)$ of EC1($\varepsilon$) ensures that except for a set $R$ of $W_1$, every vertex $w$ in $W_1\setminus R$ satisfies $\delta^0_{W_1}(w)\geq(1-10\sqrt{\varepsilon})|W_1|$ and any $u\in R$ has $\delta^0_{W_1}(u)\geq\varepsilon^{1/3}|W_1|/2$. Then, evidently, we can use some vertices from $W_1\setminus (R\cup\{u_1, u_2\})$ to connect all vertices of $R$ into a single path, denoted $P^\prime=u_2^\prime\cdots u_2$. Moreover, there exists a vertex $u\in W_1\setminus(V(P^\prime)\cup\{u_1, u_2\})$ such that $P_2\circ u\circ P^\prime:=P_0$ form a path. 
Set $W_1^\prime:=W_1\setminus V(P_1\cup P_0)$ and $P^\prime:=P_1\circ P\circ P_0$. 

\smallskip

Without loss of generality, let $xy$ be an arc in $H$. We then let $u_1$ serve as the embedding of $x$, and $u_2$ as the embedding of $y$. Let $H^-=H-xy$. Partition $W_1^\prime\cup\{u_1, u_2\}$ into two subsets $W_1^1$ and $W_1^2$ such that $|W_1^1|-|W_1^2|=m_{H^-}$. Finally, applying Lemma \ref{qaz}(i), we can find a spanning $H^-$-subdivision $H^\prime$ in $D[W_1^1\cup W_1^2]$. Clearly, the union $H^\prime\cup P^\prime$ then forms the required spanning $H$-subdivision, completing the proof of Theorem \ref{song1} for the Extremal Case $1$ with parameter $\varepsilon$ (EC1($\varepsilon$)).


\medskip

\textbf{Case a.2: $D$ belongs to EC2($\varepsilon$).} In this case, the vertex set $V(D)$ can be partitioned into three disjoint sets $W_1$, $W_2$ and $W_3$ such that for each $i\in[2]$, we have $|W_i|=(1/2-\varepsilon)n\pm\sqrt{10\varepsilon}|W_i|$, and the pair $(W_1, W_2)$ forms an $\varepsilon$-almost complete bipartite digraph. Without loss of generality, assume $|W_1|\geq|W_2|$. Let $(U_1, U_2)$ be a bipartition of $H$ satisfying the conclusion of Lemma \ref{lem:bipartition} and define $t:=|U_1|+e(U_2)-(|U_2|+e(U_1))\geq0$. 

We now construct a collection $\mathcal{P}$ of disjoint paths that passes $W_3$. The construction splits into two cases based on the balance between $W_1$ and $W_2$:

\textbf{(a.2.1)} If $|W_1|-t\geq|W_2|$, then by Lemma \ref{extremal2}, there exists a collection $\mathcal{P}$ of disjoint paths such that $|W_1\setminus V(\mathcal{P})|=|W_2\setminus V(\mathcal{P})|+t$ and passes through $W_3$.

\textbf{(a.2.2)} Otherwise, we have $|W_1|-t<|W_2|$. Since $h-(|W_1|-|W_2|)\leq 2(h/2-1)+|W_1|-|W_2|$, a procedure analogous to the proof of Lemma \ref{extremal2} also yields a collection $\mathcal{P}$ of disjoint paths such that $|W_1\setminus V(\mathcal{P})|=|W_2\setminus V(\mathcal{P})|+t$ and passes through $W_3$.

Furthermore, for $j\in[2]$ and $\sigma\in\{+, -\}$, any vertices $u, v\in W_j$ satisfy $|N^\sigma_{W_{3-j}}(u)\cap N^{-\sigma}_{W_{3-j}}(v)|\geq(1-10\sqrt{\varepsilon})|W_{3-j}|+\frac{\varepsilon^{1/3}|W_{3-j}|}{8}-|W_{3-j}|
>\frac{\varepsilon^{1/3}|W_{3-j}|}{10}>|W_3|$. This inequality guarantees robust connectivity between the sets, which ensures the existence of a specific path $P$ of the form $u_1P^1_1v_1P^2_1u_2P^3_1u_3v_2P^4_1v_3\cdots$, where for each $i$, $u_i\in W_1$, $v_i\in W_2$, and the path $P^j_i\in \mathcal{P}$ are of the following types: $P^1_i$ is a $(W_2, W_1)$-path, $P_i^2$ is a $(W_1, W_2)$-path, $P^3_i$ is a $W_1$-path and $P^4_i$ is a $W_2$-path. This path $P$ also satisfies  $|W_1\setminus V(P)|=|W_2\setminus V(P)|+t$.

We then embed $U_1$ into $W_1\setminus V(P)$ and $U_2$ into $W_2\setminus V(P)$. For simplicity, we still denote these embedded sets by $U_1$ and $U_2$, respectively. Finally, applying Lemma \ref{qaz}(i) to the remaining $\varepsilon$-almost complete bipartite digraph $(W_1 \setminus V(P), W_2 \setminus V(P))$ yields the desired spanning $H$-subdivision. In this construction, the path $P$ can naturally serve as a subpath of one of the branch paths in the subdivision.



\medskip

\textbf{Case a.3: $D$ belongs to EC3($\varepsilon, \varepsilon_1$).} In this case, $V(D)=\bigcup_{i=1}^5 W_i$ with $|W_i|\approx|W_{i+2}|$ for $i\in[2]$ satisfying the following properties: for each $i\in\{1, 3\}$, $D[W_i]$ is $\varepsilon$-almost complete; the pair $(W_2, W_4)$ is $\varepsilon$-almost bipartite complete; and for any $j\in[4]$, the pair $(W_j, W_{j+1})$ is $\varepsilon$-almost one-way complete (where the subscript of $W_{j+1}$ taken the remainder of modulo $3$ when $j=4$).

\smallskip

\noindent \textbf{Subcase a.3.1: $|W_1|\geq |W_2|$.}

By Lemma \ref{extremal3}, there exists a set $\mathcal{P}$ of disjoint $W_1$-paths, $W_3$-paths, $(W_2, W_4)$-paths and $(W_4, W_2)$-paths, such that $|W_2\setminus V(\mathcal{P})|=|W_4\setminus V(\mathcal{P})|$ and $\mathcal{P}$ passes through $W_5$. For convenience, we denote by $\mathcal{P}_i$ the set of all $W_i$-paths in $\mathcal{P}$ for $i\in\{1, 3\}$, by $\mathcal{P}_2^1$ the set of all $(W_2, W_4)$-paths, and by $\mathcal{P}_2^2$ the set of all $(W_4, W_2)$-paths. For each $i\in[4]$, let $W_i^\prime=W_i\setminus V(\mathcal{P})$.

\smallskip

\textbf{Step a.3.1  Constructing a $W_1$-path passing through $W_3^\prime\cup V(\mathcal{P}_3)$.} 

Since $D[W_3]$ is $\varepsilon$-almost complete, we can argue similarly to Case a.1 to obtain a Hamiltonian path $P=a\cdots b$ in $D[W_3^\prime\cup V(\mathcal{P}_3)]$. Utilizing the property that $(W_j, W_{j+1})$ is $\varepsilon$-almost one-way complete for any $j\in[4]$, we can find two disjoint $3$-paths $x_1y_1a$ and $by_2x_2$ with $x_1, x_2\in W_1^\prime$, $y_1\in W_2^\prime$ and $y_2^\prime\in W_4^\prime$. This implies the existence of a $W_1$-path $P_1=x_1\cdots x_2$ passing through $W_3^\prime\cup V(\mathcal{P}_3)$.

\smallskip

\textbf{Step a.3.2  Constructing a $W_1$-path passing through $W_2^\prime\cup W_4^\prime\cup V(\mathcal{P}_2^1\cup\mathcal{P}_2^2)$. }

Since $(W_2^\prime, W_4^\prime)$ is $\varepsilon$-almost complete balanced bipartite, we can find a Hamiltonian $(W_2, W_4)$-path $P^\prime$ in $D[W_2^\prime\cup W_4^\prime\cup V(\mathcal{P}_2^1\cup\mathcal{P}_2^2)]$ of the form $$P_1^1 u_1v_1P_2^1\cdots u_{k-1}v_{k-1}P^1_ku_{k+1}P^2_1v_{k+1}u_{k+2}P^2_2\cdots v_{k+s-1}u_{k+s}P^2_s,$$ where for any integer $i$, $P_i^1\in \mathcal{P}_2^1$ and $P_i^2\in \mathcal{P}_2^2$, with $u_i\in W_2^\prime$ and $v_i\in W_4^\prime$. Furthermore, as $(W_1, W_2)$ and $(W_4, W_1)$ is $\varepsilon$-almost one-way complete, there exist vertices $x_1^\prime, x_2^\prime\in W_1^\prime$ such that there is an arc from $x_1^\prime$ to the initial of $P^\prime$ and an arc from the terminal of $P^\prime$ to $x_2^\prime$. Consequently, we obtain a $W_1$-path $P_2=x_1^\prime\cdots x_2^\prime$ passing through $W_2^\prime\cup W_4^\prime\cup V(\mathcal{P}_2^1\cup\mathcal{P}_2^2)$.

\smallskip

\textbf{Step a.3.3: Constructing paths within $W_1$.}

Since $D[W_1]$ is also $\varepsilon$-almost complete, there exists a $W_1$-path $P_3=x_1^{\prime\prime}\cdots x_2^{\prime\prime}$ passing through $V(\mathcal{P}_1)$ and except vertices of $W_1^\prime$. Finally, in the $\varepsilon$-almost complete digraph $D[W_1^\prime\setminus(\{x_1, x_2^\prime\}\cup V(P_3))]$, we can find two disjoint path $P_4=x_2\cdots x_1^\prime$ and $P_5=x_2^\prime\cdots x_1^{\prime\prime}$. Let $P:=P_1\circ P_4\circ P_2\circ P_5\circ P_3$.

\smallskip

Without loss of generality, let $xy$ be an arc in $H$, We then let $x_1$ serve as the embedding of $x$, and $x_2^{\prime\prime}$ as the embedding of $y$. Let $H^-=H-xy$. Partition $(W_1\setminus V(P))\cup\{x_1, x_2^{\prime\prime}\}$ into two subsets $W_1^1$ and $W_1^2$ such that $|W_1^1|-|W_1^2|=m_{H^-}$. Finally, applying Lemma \ref{qaz}(i), we can find a spanning $H^-$-subdivision $H^\prime$ in $D[W_1^1\cup W_1^2]$. Clearly, the union $H^\prime\cup P$ then forms the required spanning $H$-subdivision.

\smallskip

\noindent \textbf{Subcase a.3.2: $|W_1|<|W_2|$.}

Let $(U_1, U_2)$ be a bipartition of $H$ satisfying the conclusion of Lemma \ref{lem:bipartition} and define $t:=|U_1|+e(U_2)-(|U_2|+e(U_1))\geq0$. Following an approach analogous to that in Case a.2, and by an argument similar to that of Lemma \ref{extremal3}, there exists a set $\mathcal{P}$ of disjoint $W_1$-paths, $W_3$-paths, $(W_2, W_4)$-paths and $(W_4, W_2)$-paths, such that $|W_2\setminus V(\mathcal{P})|=|W_4\setminus V(\mathcal{P})|+t$ and $\mathcal{P}$ passes through $W_5$. We use the same notation $\mathcal{P}_i$, $\mathcal{P}_2^1$, $\mathcal{P}_2^2$, as in Subcase a.3.1, and define $W_i^\prime=W_i\setminus V(\mathcal{P})$ for $i\in[4]$.

\smallskip

Moreover, since each $D[W_i]$ (for $i\in\{1, 3\}$) is $\varepsilon$-almost complete, there is a Hamiltonian path $Q_i=x_i\cdots y_i$ of $D[W_i^\prime\cup V(\mathcal{P}_i)]$. By the $\varepsilon$-almost one-way completeness of $(W_i, W_{i+1})$ for $i\in[4]$ (with $W_5=W_1$), we can find vertices $u_1, u_2\in W_2^\prime$ and $v_1, v_2\in W_4^\prime$ such that $\{v_1x_1, y_1u_1, u_2x_3, y_3v_2\}\subset A(D)$. This yields two disjoint paths: $Q_1^\prime=v_1P_1u_1$ passing through $W_1^\prime\cup V(\mathcal{P}_1)$, and $Q_3^\prime=u_2P_3v_2$ passing through $W_3^\prime\cup V(\mathcal{P}_3)$. 

\smallskip

We now embed $U_1$ into $W_2^\prime\setminus\{u_1, u_2\}$ and $U_2$ into $W_4^\prime\setminus\{u_1, u_2\}$ (retaining the notation $U_1$ and $U_2$ for the embedded sets, respectively). Finally, applying Lemma \ref{qaz}(i) to the remaining $\varepsilon$-almost complete bipartite digraph between these two sets produces the desired spanning $H$-subdivision. In this construction, the paths from $\mathcal{P}_2^1$ and $\mathcal{P}_2^2$, as well as the paths $Q_1^\prime$ and $Q_3^\prime$ are incorporated as subpaths of the corresponding branch paths.

\smallskip

Hence, This completes the proof of Theorem \ref{song1}.
\end{proof}
\subsection{Proof of Theorem \ref{song2}}
\begin{proof}
Let $H$ be a digraph with $h$ arcs and minimum degree $\delta(H)\geq1$. Let $C_0$ be the constant from Theorem \ref{song2}. For any integer $m$ and any integer partition $n=n_1+\cdots +n_m$, let $D$ be a digraph on $n\geq C_0hm$ vertices satisfying $\delta^0(D)\geq\frac{n+m+h}{2}-1$. Fix a hierarchy of parameters $1/C_0\ll\varepsilon^\prime\ll\varepsilon_1\ll\varepsilon\ll1$. We prove Theorem \ref{song2} under the assumption that $D$ is not $\varepsilon^\prime$-stable. If $h=1$, i.e., $H=xy$, then $m$ disjoint $H$-subdivisions correspond simply to $m$ disjoint paths. Since in this case, $\delta^0(D)\geq\frac{n+m+h}{2}-1=\frac{n+m-1}{2}\geq\frac{n-1}{2}$, by Ghouila-Houri's theorem \cite{Ghouila-Houri}, $D$ contains a Hamiltonian path. We can then partition this Hamiltonian path consecutively into $m$ disjoint segments whose orders are respectively $n_1, \ldots, n_m$. Therefore, in what follows we assume $h\geq2$. We also assume $m\geq2$, since otherwise the statement reduces to Theorem \ref{song1}, which has already been proved.

\smallskip

By Lemma \ref{claim1}, every such digraph $D$ either belongs to EC1($\varepsilon$) or EC2($\varepsilon$) or EC3($\varepsilon, \varepsilon_1$). We therefore consider the following cases.

\smallskip

\textbf{Case b.1: $D$ belongs to EC1($\varepsilon$)}. In this case, $V(D)=W_1\cup W_2\cup W_3$ such that for each $i\in[2]$, $D[W_i]$ is $\varepsilon$-almost compete. By Lemma \ref{extremal1}, there exists a set $\mathcal{P}$ of disjoint $W_1$-paths and $W_2$-paths, each of length at most $4$, and passing through $W_3$.  For $i\in[2]$, let $\mathcal{P}_{i}\subseteq\mathcal{P}$ be the set of $W_i$-paths, and define  $W_i^\prime=W_i\setminus V(\mathcal{P})$. We now distinguish two subcases.

\smallskip

First, suppose that there exists a subscript $i_0\in[m]$ such that $|W_1\cup V(\mathcal{P}_{1})|=n_1+\cdots+n_{i_0}$ (or, symmetrically, $|W_2\cup V(\mathcal{P}_{2})|=n_{i_0+1}+\cdots+n_{m}$). Since each $W_i^\prime$ (for $i\in[2]$) is $\varepsilon$-almost complete, we can partition it into two subsets $W_i^1$ and $W_i^2$, and simultaneously partition each $\mathcal{P}_{i}$ into two parts $\mathcal{P}_{i}^1$ and $\mathcal{P}_{i}^2$ 
such that $|W_i^1|+|V(\mathcal{P}_{i}^1)|=|W_i^2|+|V(\mathcal{P}_{i}^2)|+m_H$. Clearly, each induced subdigraph $D[W_i^j]$ ($i, j\in[2]$) remains $\varepsilon$-almost complete, and the pair $(W_i^1, W_i^2)$ forms an $\varepsilon$-almost complete bipartite digraph. We then proceed as follows: first, we  incorporate the exceptional vertices of $W_i^\prime$ that have lower but still sufficient semi-degree.  Afterwards, we apply Lemma \ref{qaz}(ii) to $(W_i^1, W_i^2)$ to obtain disjoint $H$-subdivisions whose orders meet the required conditions. Throughout this process, the paths in $\mathcal{P}_i$ can be used directly as segments of the branch paths of these $H$-subdivisions.

\smallskip

Secondly, we assume that no such index $i_0$ exists. Given the minimum semi-degree condition $\delta^0(D)\geq\frac{n+m+h}{2}-1$, it follows that for any vertices $u\in W_1$ and $v\in W_2$, we have that either
\begin{align*}
\begin{split}
\left \{
\begin{array}{ll}
\emph{\emph{$uv\in A(D)$ and $|N^+(u)\cap N^-(v)|\geq2\delta^0(D)-(n-1)\geq m+h-1$, or}}\\
\emph{\emph{$uv\notin A(D)$ and $|N^+(u)\cap N^-(v)|\geq 2\delta^0(D)-(n-2)\geq m+h$}},
\end{array}
\right.
\end{split}
\end{align*}
Similarly, we also get that either
\begin{align*}
\begin{split}
\left \{
\begin{array}{ll}
\emph{\emph{$vu\in A(D)$ and $|N^-(u)\cap N^+(v)|\geq2\delta^0(D)-(n-1)\geq m+h-1$, or}}\\
\emph{\emph{$vu\notin A(D)$ and $|N^-(u)\cap N^=(v)|\geq 2\delta^0(D)-(n-2)\geq m+h$}},
\end{array}
\right.
\end{split}
\end{align*}
Consequently, there exist two disjoint minimal paths $P_1$ and $P_2$, where $P_1$ is directed from $W_1$ to $W_2$, and $P_2$ from $W_2$ to $W_1$. For convenience, let $W_i^\prime=W_i\setminus V(P_1\cup P_2)$ for $i\in[2]$. Now, by an argument analogous to the proof of Lemma \ref{extremal1} (noting that during the process, $P_1$ and $P_2$ can always be suitably adjusted to remain disjoint from the final required arc set), we obtain a collection $\mathcal{P}^\prime$ of disjoint $W_1^\prime$-paths and $W_2^\prime$-paths, and passing through $W_3\setminus V(P_1\cup P_2)$.
For $i\in[2]$, let $\mathcal{P}_{i}^\prime\subseteq\mathcal{P}^\prime$ denote the subset of $W_i^\prime$-paths. 

\smallskip

Let $i_0\in[m]$ be the largest index satisfying $n_1+\cdots+n_{i_0}\leq|W_1\cup V(\mathcal{P}^\prime_1)|$, and define $r=|W_1\cup V(\mathcal{P}^\prime_1)|-(n_1+\cdots+n_{i_0})$. Recall that in the digraph $D[W_1^\prime]$, except for $10\sqrt{\varepsilon}|W_1^\prime|$ exceptional vertices, every vertex $w$ satisfies $\delta^0_{W_1^\prime}(w)\geq (1-10\sqrt{\varepsilon})|W_1^\prime|$, and these exceptional vertices have semi-degree at least $\varepsilon^{1/3}|W_1^\prime|/2$ within $W^\prime_1$. Using these properties, we can find an $r$-vertex path $P_0$ within $D[W_1^\prime]$ connecting the terminal of $P_2$ to the initial of $P_1$. Define $W_2$-path $P:=P_2\circ P_0\circ P_1$.

\smallskip 

We now proceed similarly to the previous case. Since each $W_i^\prime\setminus V(P)$ (for $i\in[2]$) is $\varepsilon$-almost complete, we can partition it into two subsets $W_i^1$ and $W_i^2$. Simultaneously, we  partition each $\mathcal{P}_{i}^\prime$ into two subcollections $\mathcal{P}_{i}^1$ and $\mathcal{P}_{i}^2$
such that $|W_1^1|+|V(\mathcal{P}_{i}^1)|=|W_1^2|+|V(\mathcal{P}_{i}^2)|+m_H$, and  $|W_2^1|+|V(\mathcal{P}_{i}^1)|=|W_1^2|+|V(\mathcal{P}_{i}^2)|+|V(P)|+m_H$. Clearly, each induced subdigraph $D[W_i^j]$ ($i, j\in[2]$) remains $\varepsilon$-almost complete, and each pair $(W_i^1, W_i^2)$ forms an $\varepsilon$-almost complete bipartite digraph. Then within the subdigraph $D[W_1^\prime\cup V(\mathcal{P}_{1}^\prime)\setminus V(P)]$, we obtain disjoint $H$-subdivisions of orders $n_1, \ldots, n_{i_0}$, by first using the exceptional vertices of $W_1^\prime$, and then applying Lemma \ref{qaz}(ii) to $(W_1^1, W_1^2)$. Throughout this construction, paths from $\mathcal{P}_{1}$ can be used directly as segments of these $H$-subdivisions. Similarly, within the subdigraph $D[W_1^\prime\cup V(\mathcal{P}_{1}^\prime)\cup V(P)]$, we obtain disjoint $H$-subdivisions of orders $n_{i_0}, \ldots, n_{m}$, by first using exceptional vertices of $W_2^\prime$ and then applying Lemma \ref{qaz}(ii) to $(W_2^1, W_2^2)$. Here, the path $P$ and the paths from $\mathcal{P}_{1}$ can be directly used to construct some paths of these $H$-subdivisions.

\medskip

\textbf{Case b.2: $D$ belongs to EC2($\varepsilon$)}. In this case, we have that $V(D)=W_1\cup W_2\cup W_3$ such that for each $i\in[2]$, $|W_i|=(1/2-\varepsilon)n\pm\sqrt{10\varepsilon}|W_i|$, and $(W_1, W_2)$ forms an $\varepsilon$-almost complete bipartite digraph. By the minimum semi-degree condition of $D$, for any $u\in W_1$ and $v\in W_2$, we have
\begin{align}\label{all}
|N^+(u)\cap N^-(v)|\geq2\delta^0(D)-(n-1)\geq m+h-1.
\end{align}
Let $(U_1, U_2)$ be a bipartition of $H$ satisfying the conclusion of Lemma \ref{lem:bipartition} and define $t:=|U_1|+e(U_2)-(|U_2|+e(U_1))\geq0$. We proceed according to the parity of $m$ and $t$.

\smallskip

\noindent \textbf{Subcase b.2.1: $m$ is even, and t is odd.}

For all even $n_i$, we use $(\ref{all})$ to obtain a collection $\mathcal{P}$ of $l$ disjoint $3$-paths from $W_1$ to $W_2$, where $l=\sum_{i=1}^m(1-(\lceil\frac{n_i}{2}\rceil-\lfloor\frac{n_i}{2}\rfloor))$.  Next, following an argument similar to the proof of Lemma \ref{extremal2} (note that in this setting the degree calculated in (\ref{22m}) decreases by at most $2m$, and the remaining degree is at least $n-m+h-2$, which still meets the condition of Lemma \ref{extremal2}), we obtain a family $\mathcal{P}^\prime$ of disjoint paths covering all remaining vertices of $W_3$ and satisfying $|W_1\setminus V(\mathcal{P}\cup\mathcal{P}^\prime)|=|W_2\setminus V(\mathcal{P}\cup\mathcal{P}^\prime)|$. Finally, applying the method of Lemma \ref{qaz}(ii) to the $\varepsilon$-almost complete bipartite digraph $(W_1\setminus V(\mathcal{P}\cup\mathcal{P}^\prime), W_2\setminus V(\mathcal{P}\cup\mathcal{P}^\prime))$, yields yields the required set of disjoint $H$-subdivisions. In this construction, each $H$-subdivision of even order uses exactly one path from $\mathcal{P}$, while paths from $\mathcal{P}^\prime$ may be incorporated as subpaths in the largest subdivision.

\smallskip

\noindent \textbf{Subcase b.2.2: Both $m$ and $t$ are even.}

The argument differs from Subcase b.2.1 only in the initial step: here we apply $(\ref{all})$ to obtain a collection $\mathcal{P}$ of $l$ disjoint $3$-paths, where $l=\sum_{i=1}^m(\lceil\frac{n_i}{2}\rceil-\lfloor\frac{n_i}{2}\rfloor)$. Each resulting $H$-subdivision of odd order then uses exactly one from $\mathcal{P}$. The remaining steps are identical to Subcase b.2.1.

\smallskip

\noindent \textbf{Subcase b.2.3: Both $m$ and $t$ are odd.}

As in Subcase b.2.1, we use $(\ref{all})$ to obtain a collection $\mathcal{P}$ of $l$ disjoint $3$-paths with $l=\sum_{i=1}^m(1-(\lceil\frac{n_i}{2}\rceil-\lfloor\frac{n_i}{2}\rfloor))$. Then, similarly to the proof of Lemma \ref{extremal2}, we find a family $\mathcal{P}^\prime$ of disjoint paths covering all remaining vertices of $W_3$ such that  $|W_1\setminus V(\mathcal{P}\cup\mathcal{P}^\prime)|=|W_2\setminus V(\mathcal{P}\cup\mathcal{P}^\prime)|+t$. Applying the proof of Lemma \ref{qaz}(ii) to the $\varepsilon$-almost complete bipartite digraph $(W_1\setminus V(\mathcal{P}\cup\mathcal{P}^\prime), W_2\setminus V(\mathcal{P}\cup\mathcal{P}^\prime))$, gives the required disjoint $H$-subdivisions, with each even-order subdivision using exactly one path from $\mathcal{P}$, and paths from $\mathcal{P}^\prime$ usable in the largest subdivision.

\smallskip

\noindent \textbf{Subcase b.2.4: $m$ is odd and $t$ is even.}

This case mirrors Subcase b.2.3: we apply $(\ref{all})$ to obtain a collection $\mathcal{P}$ of $l$ disjoint $3$-paths with $l=\sum_{i=1}^m(\lceil\frac{n_i}{2}\rceil-\lfloor\frac{n_i}{2}\rfloor)$. Each $H$-subdivision of odd order then uses exactly one from $\mathcal{P}$. The remaining steps are identical to Subcase b.2.3.

\smallskip

Having established Subcases b.2.1 through b.2.4, we conclude that Theorem \ref{song2} holds when $D$ belongs to EC2($\varepsilon$).

\medskip

\textbf{Case b.3: $D$ is EC3($\varepsilon, \varepsilon_1$)}. In this case, the vertex set admits a partition $V(D)=\bigcup_{i=1}^5 W_i$ with $|W_i|\approx|W_{i+2}|$ for $i\in[2]$ and the following properties hold. For $i\in\{1, 3\}$, the induced subdigraph $D[W_i]$ is $\varepsilon$-almost complete; The pair $(W_2, W_4)$ forms an $\varepsilon$-almost complete bipartite digraph; For $j\in[4]$ (where the index of $W_{j+1}$ is taken modulo $4$ when $j=4$), the pair $(W_j, W_{j+1})$ is $\varepsilon$-almost one-way complete.

\smallskip

By Lemma \ref{extremal3}, there is a collection $\mathcal{P}$ of disjoint paths, specifically, $W_1$-paths, $W_3$-paths, $(W_2, W_4)$-paths and $(W_4, W_2)$-paths, that passes through $W_5$ and satisfies  $|W_2\setminus V(\mathcal{P})|=|W_4\setminus V(\mathcal{P})|$. Denote by $\mathcal{P}_{1}$, $\mathcal{P}_{2}^1$, and $\mathcal{P}_{2}^2$, $\mathcal{P}_{3}$ the subsets of $\mathcal{P}$ consisting of $W_1$-paths, $(W_2, W_4)$-paths, $(W_4, W_2)$-paths, and $W_3$-paths, respectively. 
Since $(W_1, W_2)$ and $(W_4, W_1)$ are $\varepsilon$-almost one-way complete, we can proceed as in Claim  \ref{321321} of Lemma \ref{extremal3}: for each path $P\in\mathcal{P}_{2}^1$, we can choose a pair $(u, v)$ distinct vertices in $W_1\setminus V(\mathcal{P}_{1})$ such that the concatenation $P^\prime:=u\circ P\circ v$ is a $W_1$-path. Let $\mathcal{P}_{1}^{\prime}$ be the set of these extended paths. Similarly, since  $(W_2, W_3)$ and $(W_3, W_4)$ are $\varepsilon$-almost one-way complete, for each path $Q\in \mathcal{P}_{2}^2$, we can select a pair $(u^\prime, v^\prime)$ in $W_3\setminus V(\mathcal{P}_{3})$ such that $Q^\prime:=u^\prime\circ Q\circ v^\prime$ is a $W_3$-path. Denote the collection of these paths by $\mathcal{P}_{3}^{\prime}$. Note that all paths in $\mathcal{P}_{1}^{\prime}\cup\mathcal{P}_{3}^{\prime}$ are disjoint. After these extensions we still have $|W_2\setminus V(\mathcal{P}_{1}^{\prime}\cup \mathcal{P}_{3}^{\prime})|=|W_4\setminus V(\mathcal{P}_{1}^{\prime}\cup \mathcal{P}_{3}^{\prime})|$. For $i\in[4]$, define $W_i^\prime=W_i\setminus V(\mathcal{P}_{1}\cup\mathcal{P}_{1}^{\prime}\cup \mathcal{P}_{3}^{\prime}\cup\mathcal{P}_{3})$. Clearly, $(W_2^\prime, W_4^\prime)$ is an $\varepsilon$-almost complete balanced bipartite digraph. 
From the definition of the class EC3($\varepsilon, \varepsilon_1$), we have that for $i\in[4]$, $|W_i|>\varepsilon^{1/3}n-\sqrt{10\varepsilon}n$, and $|W_5|\leq4\sqrt{10\varepsilon}n$. Consequently, for $i\in[4]$, 
\begin{align*}
|W_i^\prime|\geq\varepsilon^{1/3}n-2\sqrt{10\varepsilon}n-3|W_5|\geq\varepsilon^{1/3}n/2.
\end{align*}

Recall that $m_H=\min_{U_1\cup U_2=V(H)}\big||U_1|+e(U_2)-(|U_2|+e(U_1))|$, where the minimum is taken over all bipartitions $(U_1, U_2)$ of the digraph $H$, and that $n=n_1+\cdots+n_m$. Suppose first there exist a balanced subdigraph $(W_2^0, W_4^0)$ of $(W_2, W_4)$ and a subset $I\subseteq[m]$ such that $|W_1^\prime|+|W_2^0|+|W_4^0|+|V(\mathcal{P}_{1}\cup\mathcal{P}_{1}^{\prime})|=\sum_{i\in I}n_i.$ 
(If this holds, symmetry gives $|W_3^\prime|+|W_2^\prime\setminus W_2^0|+|W_4^\prime\setminus W_4^0|+|V(\mathcal{P}_{3}^{\prime}\cup\mathcal{P}_{3})|=\sum_{i\in [m]\setminus I}n_i$). Then we directly proceed to step b.3. If no such pair $(W_2^0, W_4^0)$ exists, then without loss of generality, we can select a balanced subdigraph $(W_2^{\prime\prime}, W_4^{\prime\prime})$ of $(W_2, W_4)$ and a subset $I_0\subseteq[t]$ such that $s:=|W_1^\prime|+|W_2^{\prime\prime}|+|W_4^{\prime\prime}|+|V(\mathcal{P}_{1}\cup\mathcal{P}_{1}^\prime)|-\sum_{i\in I_0}n_i\geq3$ and $s\leq\max_{i\in[t]\setminus I_0}n_i-m_H$. Since each pair $(W_i^\prime, W_{i+1}^\prime)$ is $\varepsilon$-almost one-way complete and $D[W_1^\prime]$ is $\varepsilon$-almost complete, we can find a $W_3^\prime$-path $P_0$ that uses exactly $s-2$ vertices of $W_1^\prime$, one vertex of $W_2^\prime$, and one vertex of $W_4^\prime$. For consistency, we now update the sets as follows: $I:=I_0$, $W_1^\prime:=W_1^\prime\setminus V(P_0)$, $W_2^0:=W_2^{\prime\prime}\setminus V(P_0)$, $W_4^0:=W_4^{\prime\prime}\setminus V(P_0)$, and $W_3^\prime:=W_3^\prime\setminus V(P_0)$, $\mathcal{P}_{3}^\prime:=\mathcal{P}_{3}^\prime\cup P_0$. Hence for each case, we always have
\begin{align}
&|W_1^\prime|+|W_2^0|+|W_4^0|+|V(\mathcal{P}_{1}\cup\mathcal{P}_{1}^{\prime})|=\sum_{i\in I}n_i.\\
&|W_3^\prime|+|W_2^\prime\setminus W_2^0|+|W_4^\prime\setminus W_4^0|+|V(\mathcal{P}_{3}^{\prime}\cup\mathcal{P}_{3})|=\sum_{i\in [m]\setminus I}n_i.
\end{align}
Clearly, the pairs $(W_2^0, W_4^0)$ and $(W_2^\prime\setminus W_2^0, W_4^\prime\setminus W_4^0)$ are $\varepsilon$-almost complete balanced bipartite, and $(W_1^\prime, W_2^0)$, $(W_4^0, W_1^\prime)$, $(W_2^\prime\setminus W_2^0, W_3^\prime)$ and $(W_3^\prime, W_4^\prime\setminus W_4^0)$ are $\varepsilon$-almost one-way complete. Also, for each $i\in\{1, 3\}$, the induced subdigraph $D[W_i^\prime]$ is $\varepsilon$-almost complete. More precisely, for each $i$,
except for a set $R_i$ of $W_i$, every vertex $w$ in $W_i\setminus R_i$ satisfies $\delta^0_{W_i}(w)\geq(1-20\sqrt{\varepsilon})|W_i|$ and any $u\in R$ has $\delta^0_{W_i}(u)\geq\varepsilon^{1/3}|W_i|/10$.
Without loss of generality, (by reordering indices if needed) let $I={1,\ldots, |I|}$. Finally, we proceed to Step b.3.

\smallskip

\textbf{Step b.3.} Using the properties above, we can find $|I|$ disjoint $W_1^\prime$-paths $P_1, \ldots, P_{|I|}$, where each $P_j=u_j\cdots v_j$ has order $|V(P_j)|:=n_j^\prime\leq n_j-m_H$. Similarly, we can find $m-|I|$ disjoint $W_3^\prime$-paths $P_{|I|+1}, \ldots, P_{m}$, where each $P_j=u_j\cdots v_j$ has order $|V(P_j)|:=n_j^\prime\leq n_j-m_H$. Let $\widetilde{W_1}$ and $\widetilde{W_3}$ denote the vertices of $W_1^\prime$ and $W_3^\prime$ respectively, not used in these paths.

\smallskip

Now, let $xy$ be an arbitrary arc of $H$ and set $H^-=H-xy$. For each $j\in[m]$, we will use $u_j$ as the embedding of vertex $x$, and $u_j$ as the embedding of vertex $y$. Partition $\widetilde{W_1}\cup\bigcup_{j\in I}(u_j, v_j)$ into two subsets $W_1^1$ and $W_2^2$ such that $|W_1^1|-|W_1^2|=m_{H^-}$. Similarly, partition $\widetilde{W_3}\cup\bigcup_{j\in [m]\setminus I}(u_j, v_j)$ into two subsets $W_3^1$ and $W_3^2$ such that $|W_3^1|-|W_3^2|=m_{H^-}$.

\smallskip

Finally, we apply Lemma \ref{qaz}(ii) to $(W_1^1, W_1^2)$. This yields $|I|$ disjoint $H^-$-subdivisions $H_1^-, \ldots, H_{|I|}^-$ within $D[W_1^\prime\cup W_2^0\cup W_4^0\cup V(\mathcal{P}_{1}\cup\mathcal{P}_{1}^{\prime})]$. For each $j\in I$, the union $H_j^-\cup P_j$ then forms a  required $H$-subdivision of order $n_j$. Applying Lemma \ref{qaz}(ii) to $(W_3^1, W_3^2)$ similarly yields $m-|I|$ disjoint $H^-$-subdivisions $H_{|I|+1}^-, \ldots, H_{m}^-$ within $D[W_3^\prime\cup(W_2\setminus W_2^0)\cup(W_4\setminus W_4^0)\cup V(\mathcal{P}_{3}\cup\mathcal{P}_{3}^{\prime})]$. For each $j\in [m]\setminus I$, the union $H_j^-\cup P_j$ forms a required $H$-subdivision of order $n_j$.

\smallskip

Hence, Theorem \ref{song2} holds when $D$ belongs to EC3($\varepsilon, \varepsilon_1$). Consequently, this completes the proof of Theorem \ref{song2}.
\end{proof}

\section{Concluding remarks}
Recall that for a digraph $H$ and any bipartition $(U_1, U_2)$ of its vertex set, we have $m_H=\min_{U_1\cup U_2=V(H)}|(|U_1|+e(U_2))-(|U_2|+e(U_1))|$. Following the same line of proof as in our paper, the minimum semi-degree conditions in Theorems \ref{song1} and \ref{song2} can in fact be strengthened to $\delta^0(D)\geq\frac{n}{2}+m_H$ and $\delta^0(D)\geq\frac{n+m}{2}+m_H$, respectively. Since the argument carries over almost verbatim, we leave the verification to the interested reader. We also note that in some cases, the parameter $m_H$ provides a better bound than $h/2-1$. For instance, if $H$ is a balanced bipartite digraph with at least three arcs, then we have $m_H=0$, which typically yields a tighter estimate.

\smallskip

For a graph $H$ with edge labels $e_1, \ldots, e_{|E(H)|}$ and an integer set $\mathcal{N}=\{n_1, \ldots, n_{|E(H)|}\}$, an \emph{$(\mathcal{N}H)$-subdivision} is obtained from $H$ by subdividing each edge $e_i$ into a path of the length $n_i$.

Coll, Magnant and Nowbandegani \cite{Coll} studied spanning $H$-subdivision with prescribed path lengths under a degree sum condition. Specifically, they showed that for a graph $H$ with $h$ edges, $\delta(H)\geq1$, and vertices labeled $v_1, \ldots, v_{|V(H)|}$, there exists $n_0>0$ such that the following holds: for any set of integers $\mathcal{N}=\{n_1, \ldots, n_h\}$ with $n_i\geq n_0$ for all $i$, and any $n$-vertex graph $G$ with $\sigma_2(G)\geq n+2h-1$, if $x_1, \ldots, x_{|V(H)|}$ are distinct vertices in $V(G)$, then $G$ contains a spanning $(\mathcal{N}H)$-subdivision in which each vertex $x_j$ corresponds to $v_j$. This result improves earlier work of Chizmar et al. \cite{Chizmar}.

It is natural to ask analogous questions for digraphs. We therefore propose the following.
\begin{conjecture}\label{pppp}
Let $H$ be a digraph with $h$ arcs and $\delta(H)\geq1$. There is a constant $n_0$ and a real number $C_H$,  depending only on $H$ such that the following holds. If $D$ is an $n$-vertex digraph satisfying  $d^+_D(u)+d^-_D(v)\geq n+C_H$ for every pair of vertices $u, v\in V(D)$ with $uv\notin A(D)$, then for any set of integers $\mathcal{N}=\{n_1, \ldots, n_{h}\}$ with $n_i\geq n_0$ for all $i$, the digraph $D$ contains a spanning $(\mathcal{N}H)$-subdivision.
\end{conjecture}
\begin{problem}
In the context of Conjecture \ref{pppp}, given an integer set $\mathcal{N}=\{n_1, \ldots, n_k\}$, is it always possible to find a perfect $H$-subdivision tiling in $D$, where the orders of the $H$-subdivisions are $n_1, \ldots, n_k$, respectively?
\end{problem}
\section*{ACKNOWLEDGMENTS}
\textsuperscript{1} School of Mathematics, Shandong University, Jinan 250100, China. \textsuperscript{2} Mathematical Institute, University of Oxford, Oxford OX1 2JD, UK. Zhilan Wang and Jin Yan are supported by National Natural Science Foundation of China (No.12571373, 12071260), and Yangyang Cheng is supported by a PhD studentship of ERC Advanced Grant 883810.

\section*{A \quad Proof of Lemma \ref{claim1}}
We use structural analysis to prove Lemma \ref{claim1}.

\smallskip

\emph{Proof of Lemma \ref{claim1}}. Since $D$ satisfies EC($\varepsilon^\prime$), there exist two (not necessarily disjoint) vertex sets $U_1$ and $U_2$ with $|U_i|\geq(1/2-\varepsilon^\prime)n$ for every $i\in[2]$, and $e^+(U_1, U_2)\leq(\varepsilon^\prime n)^2$. We consider the case by case based on the cardinality of $U_1\cap U_2$. For convenience, let $U_0:=U_1\cap U_2$.


\smallskip

\emph{Case A.1}: $|U_0|\leq\varepsilon_1 n$. In this case, we will prove that $D$ belongs to EC1($\varepsilon$). First, we define $W_i=U_i\setminus U_0$ for each $i\in[2]$, and $R=V(D)\setminus(W_1\cup W_2)$. Clearly $W_1$ and $W_2$ are disjoint and $e^+(W_1, W_2)\leq e^+(U_1, U_2)\leq(\varepsilon^\prime n)^2$. Additionally, for every $i\in[2]$, since $\varepsilon^\prime\ll\varepsilon_1\ll\varepsilon$ and $|U_1\cap U_2|=|U_0|\leq\varepsilon_1 n$, we have that $|W_i|=|U_i\setminus(U_1\cap U_2)|\geq(1/2-\varepsilon^\prime-\varepsilon_1)n\geq(1/2-\varepsilon)n$.
Further, together with $\delta^0(D)\geq(n+1)/2$, $|W_1|\leq(1/2+\varepsilon)n$, $|V(D)\setminus(W_1\cup W_2)|\leq2\varepsilon n$, $e^+(W_1, \overline{W_1})=e^+(W_1, W_2)+e^+(W_1, D\setminus(W_1\cup W_2))$, and $\varepsilon^\prime\ll\varepsilon_1\ll\varepsilon$, we can deduce that
\begin{equation}
\begin{split}\label{eqq}
e(W_1)\geq\sum_{u\in W_1}d^{+}(u)-e^+(W_1, \overline{W_1})&\geq|W_1|\cdot (n+1)/2-(\varepsilon^\prime n)^2-(1/2+\varepsilon)n\cdot2\varepsilon n\\
&\geq|W_1|^2-2\varepsilon n^2.
\end{split}
\end{equation}
Following the same calculation as in $(\ref{eqq})$, we can sum in-degrees of vertices in $W_2$ to obtain that $e(W_2)|\geq|W_2|^2-2\varepsilon n^2$.

\smallskip

Further, for $i\in[2]$, let $E_i$ be the set of vertices $u$ in $W_i$ such that $d^\sigma_{W_i}(u)\leq\varepsilon^{1/3}|W_i|$ for some $\sigma\in\{+, -\}$. Then we have that $|E_i|\leq \sqrt{10\varepsilon}|W_i|$. Further, if there exists a vertex $x$ in $E_1\cup R$ (resp., a vertex $y$ in $E_2\cup R$) such that for each $\sigma\in\{+, -\}$, $d_{W_2}^\sigma(x)>\varepsilon^{1/3}|W_2|$ (resp., $d_{W_1}^\sigma(y)>\varepsilon^{1/3}|W_1|)$, then we put $x$ (resp., $y$) into the vertex set $W_2$ (resp., $W_1$) and update the vertex sets $W_1$ and $W_2$. Repeat the above operation until there are no such vertices $x$ and $y$. Let $W_i:=W_i\setminus E_i$ for each $i\in[2]$ and $W_3=V(D)\setminus (W_1\cup W_2)$. Obviously, $|W_i|=(1/2-\varepsilon)n\pm\sqrt{10\varepsilon}|W_i|$ for $i\in[2]$ and $|W_3|\leq2\sqrt{10\varepsilon}\cdot\max\{|W_1|, |W_2|\}$, and by using the lower bound of $\delta^0(D)$, we can get that $W_1, W_2$ and $W_3$ satisfy the properties $(1.1)$ and $(1.2)$ of EC1($\varepsilon$).

\smallskip

\emph{Case A.2}: $|U_0|\geq(1/2-\varepsilon_1)n$. For this case, we will prove that $D$ belongs to EC2($\varepsilon$). First, let $W_1=U_0$, $W_2=V(D)\setminus U_0$. Then we have that $|W_2|\geq(1/2-\varepsilon^\prime)n$ since $V(D)=W_1\cup W_2$ and $|W_1|=|U_0|\leq|U_i|\leq|V(D)\setminus U_{3-i}|\leq(1/2+\varepsilon^\prime)n$, where $i\in[2]$.
Further, it is evident that $e(W_1)=e(U_0)\leq e(U_1, U_2)\leq(\varepsilon^\prime n)^2$. Combining with $\delta^0(D)\geq(n+1)/2$ and $|W_2|\leq(1/2+\varepsilon_1)n$ due to $|V(D)\setminus W_2|=|W_1|=|U_0|\geq(1/2-\varepsilon_1)n$, and $\varepsilon^\prime\ll\varepsilon_1\ll\varepsilon$, we can conclude that
\begin{align}\label{eqq1}
e^+(W_1, W_2)\geq|W_1|\cdot(n+1)/2-(\varepsilon^\prime n)^2&=|W_1|\cdot(1/2+\varepsilon_1)n-|W_1|\cdot\varepsilon_1n-(\varepsilon^\prime n)^2\nonumber\\
&\geq|W_1|\cdot|W_2|-(\varepsilon_1+(\varepsilon^\prime)^2)n^2\nonumber\\
&\geq|W_1|\cdot|W_2|-\varepsilon n^2
\end{align}
Similar to $(\ref{eqq1})$, by calculating the sum of in-degrees of vertices in $W_1$, we can also obtain that $e^+(W_2, W_1)\geq|W_1|\cdot|W_2|-\varepsilon n^2$.

\smallskip

Continually, we denote by $E_1$ (resp., $E_2$) be the set of vertices $u$ in $W_1$ (resp., $W_2$) such that $b_{W_2}(u)\leq\varepsilon^{1/3}|W_2|$ (resp., $b_{W_1}(u)\leq\varepsilon^{1/3}|W_1|$). Then we have that $|E_i|\leq \sqrt{10\varepsilon}|W_i|$. Further, if there exists a vertex $x$ in $E_1$ (resp., a vertex $y$ in $E_2$) such that $b_{W_2}(x)>\varepsilon^{1/3}|W_2|$ (resp., $b_{W_1}(y)>\varepsilon^{1/3}|W_1|)$, then we put $x$ (resp., $y$) into the vertex set $W_2$ (resp., $W_1$) and update the vertex sets $W_1$ and $W_2$. Repeat the above operation until there are no such vertices $x$ and $y$. Finally, let $W_i:=W_i\setminus E_i$ for any $i\in[2]$ and $W_3=V(D)\setminus(W_1\cup W_2)$. Obviously, $|W_i|=(1/2-\varepsilon)n\pm\sqrt{10\varepsilon}|W_i|$ for $i\in[2]$ and $|W_3|\leq2\sqrt{10\varepsilon}\cdot\max\{|W_1|, |W_2|\}$, and we see, using the lower bound of $\delta^0(D)$, that $W_1, W_2$ and $W_3$ satisfy the properties $(2.1)$ and $(2.2)$ of EC2($\varepsilon$).

\smallskip

\emph{Case A.3}: $\varepsilon_1n<|U_0|<(1/2-\varepsilon_1)n$. In this case, we will declare that $D$ belongs to EC3($\varepsilon, \varepsilon_1$). Let $W_1=U_1\setminus U_0$, $W_2=V(D)\setminus(U_1\cup U_2)$, $W_3=U_2\setminus U_0$ and $W_4=U_0$. We first estimate the cardinalities of vertex sets $W_1$-$W_4$. The following conclusion is true.
\begin{claim}\label{ccclaim}
For each $j\in\{1, 3\}$, $(1/2-\varepsilon^\prime/2)n\leq|W_j|+|W_2|\leq(1/2+\varepsilon^\prime)n$.
\end{claim}
\begin{proof}
On the one hand, since $e(W_4)+
e^+(W_4, W_3)=e(U_0)+
e^+(U_0, U_2\setminus U_0)\leq e^+(U_1, U_2)\leq(\varepsilon^\prime n)^2$, by calculating the out-degrees of vertices in $W_4$, we have that
\begin{align*}
(n+1)/2\cdot|W_4|\leq \sum_{w\in W_4}d^{+}(w)&=e^+(W_4, W_1)+e(W_4)+
e^+(W_4, W_3)+e^+(W_4, W_2) \nonumber \\
&\leq e^+(W_4, W_1)+(\varepsilon^\prime n)^2+e^+(W_4, W_2) \nonumber \\
&\leq|W_4|\cdot|W_1|+(\varepsilon^\prime n)^2+|W_4|\cdot|W_2|.
\end{align*}
Since $|W_4|=|U_0|>\varepsilon_1 n$ and $\varepsilon^\prime\ll \varepsilon_1$, this implies that $(1/2-\varepsilon^\prime/2)n\leq|W_1|+|W_2|$. Similarly, by calculating in-degrees of vertices of $W_4$, we also obtain that $(1/2-\varepsilon^\prime/2)n\leq|W_2|+|W_3|$.

On the other hand, due to $|U_i|\geq(1/2-\varepsilon^\prime)n$ for each $i\in[2]$, it can be deduced that $|W_1|+|W_2|=|V(D)\setminus U_2|\leq(1/2+\varepsilon^\prime)n$ and $|W_2|+|W_3|=|V(D)\setminus U_1|\leq(1/2+\varepsilon^\prime)n$.
Hence, the claim holds.
\end{proof}
In the following, we will first prove that $|W_1|$ is approximately equal to $|W_3|$, and similarly, $|W_2|$ is approximately equal to $|W_4|$.
\begin{claim}\label{szw1}
$-3\varepsilon^\prime n/2\leq|W_1|-|W_3|\leq3\varepsilon^\prime n/2$ and $-\varepsilon^\prime n\leq|W_2|-|W_4|\leq2\varepsilon^\prime n$.
\end{claim}
\begin{proof}
By Claim \ref{ccclaim}, we get that for each $j\in\{1, 3\}$,
\begin{align*}
(1-\varepsilon^\prime)n/2-|W_2|\leq|W_j|\leq(1/2+\varepsilon^\prime)n-|W_2|.
\end{align*}
Hence, we have that
\begin{align}\label{yyyy1}
-3\varepsilon^\prime n/2\leq|W_1|-|W_3|\leq3\varepsilon^\prime n/2.
\end{align}
Also, by Claim \ref{ccclaim} again, we have $(1/2-\varepsilon^\prime/2)n-|W_3|\leq|W_2|\leq(1/2+\varepsilon^\prime)n-|W_3|$, and
\begin{align*}
(1/2-\varepsilon^\prime/2)n+|W_3|\leq|W_1|+|W_2|+|W_3|\leq (1/2+\varepsilon^\prime)n+|W_3|.
\end{align*}
Together with $n=|W_1|+|W_2|+|W_3|+|W_4|$, this suggests that
\begin{align*}
(1/2-\varepsilon^\prime)n-|W_3|\leq|W_4|\leq(1+\varepsilon^\prime)n/2-|W_3|.
\end{align*}
Hence, we obtain that
\begin{align}\label{yyyy2}
-\varepsilon^\prime n\leq|W_2|-|W_4|\leq2\varepsilon^\prime n.
\end{align}
Inequalities $(\ref{yyyy1})$ and $(\ref{yyyy2})$ imply that $|W_1|\approx|W_3|$ and $|W_2|\approx|W_4|$.
\end{proof}
We then estimate the cardinality of the vertex set $W_i$ for each $i\in[4]$. The following claim holds.
\begin{claim}\label{szw2} We declare that the statements holds as follows.\\
$(i)$ $\varepsilon_1n/2<|W_j|<(1/2-3\varepsilon_1/4)n\ \mbox{for}\ j\in\{1, 3\}.$\\
$(ii)$ $\varepsilon_1 n/2<|W_i|<(1/2-\varepsilon_1/4)n,\ \mbox{for each}\ i\in\{2, 4\}.$
\end{claim}
\begin{proof}
Since $|W_4|=|U_0|$ and $\varepsilon_1n/2<|U_0|<(1/2-\varepsilon_1)n$, clearly,
\begin{align}\label{szw3}
\varepsilon_1n<|W_4|<(1/2-\varepsilon_1/4)n. \end{align}
In the following, we will estimate the upper and lower bounds of $|W_j|$ for $j\in\{1, 3\}$. Since $|U_1|, |U_2| \geq(1/2-\varepsilon^\prime)n$ and $\varepsilon^\prime\ll\varepsilon_1$, we have $|W_j|\geq(1/2-\varepsilon^\prime)n-|W_4|>(1/2-\varepsilon^\prime)n-(1/2-\varepsilon_1)n\geq\varepsilon_1n/2$. Also by Claim \ref{ccclaim}, $V(D)=W_1\cup W_2\cup W_3\cup W_4$ and $|W_4|=|U_0|>\varepsilon_1n$, we can deduce that
\begin{align}\label{aaal}
|W_j|&=|V(D)|-(|W_{j+2}|+|W_2|)-|W_4|\nonumber\\
&< n-(1/2-\varepsilon^\prime/2)n-\varepsilon_1 n\leq(1/2-3\varepsilon_1/4)n,
\end{align}
where the subscript of $W_{j+2}$ takes the remainder of modulo $4$. So we conclude that
\begin{align}\label{W1}
\varepsilon_1n/2<|W_j|<(1/2-3\varepsilon_1/4)n\ \mbox{for}\ j\in\{1, 3\}.
\end{align}

\smallskip

Next, we will estimate the upper and lower bounds of $|W_2|$. On the one hand, due to  $|U_1|\geq(1/2-\varepsilon^\prime)n$, $|W_2|=|V(D)|-(|W_1|+|W_3|+|W_4|)=|V(D)|-|U_1\cup U_2|$, and $\varepsilon^\prime\ll\varepsilon_1$, we have that
\begin{align*}\label{eqqq1}
|W_2|=n-(|U_1|+|U_{2}\setminus U_0|)
<n-((1/2-\varepsilon^\prime)n+\varepsilon_1n/2)<(1/2-\varepsilon_1/4)n.
\end{align*}
On the other hand, by Claim \ref{ccclaim} again and (\ref{aaal}), we obtain that
\begin{align*}
|W_2|\geq(1/2-\varepsilon^\prime/2)n-|W_1|>(1/2-\varepsilon^\prime/2)n-(1/2-3\varepsilon_1/4)n\geq\varepsilon_1 n/2.
\end{align*}
Together with (\ref{szw3}), we can conclude that
\begin{equation}\label{W2}
\begin{split}
\varepsilon_1 n/2<|W_i|<(1/2-\varepsilon_1/4)n,\ \mbox{for each}\ i\in\{2, 4\}.
\end{split}
\end{equation}
Therefore, we have successfully proven this claim.
\end{proof}
It is worth noting that if there exists an $i\in\{1, 3\}$ such that $|W_i|<\varepsilon^{1/3}n$, then we treat the vertices in $W_1\cup W_3$ as exceptional vertices and set $(W_2, W_4):=(W_1, W_2)_{EC2(\varepsilon^\prime)}$. In this case, by appropriately choosing new parameters, the situation can be reduced to the extremal case EC2($\varepsilon$). Similarly, if there exists a $j\in\{2, 4\}$ such that $|W_j|<\varepsilon^{1/3}n$, then we treat the vertices in $W_2\cup W_4$ as exceptional vertices and set $(W_1, W_3):=(W_1, W_2)_{EC1(\varepsilon)}$. Again, with suitable adjustment of parameters, this can be reduced to the extremal case EC1($\varepsilon$). Therefore, for all $i\in[4]$, we have $|W_i|\geq\varepsilon^{1/3}n$.

\smallskip

In what follows, we will prove that properties $(3.1)$, $(3.2)$, $(3.3)$ and $(3.4)$ of EC3($\varepsilon, \varepsilon_1$) hold. Firstly, it yields from $e^+(U_1, U_2)\leq(\varepsilon^\prime n)^2$ that, for each $j\in\{1, 4\}$, $e^+(W_j, W_3\cup W_4)\leq e^+(U_1, U_2)\leq(\varepsilon^\prime n)^2$. By Claim \ref{ccclaim}, $\delta^0(D)\geq(n+1)/2$ and $\varepsilon^\prime\ll \varepsilon$, this implies that
\begin{align}\label{W3}
e^+(W_j, W_1\cup W_2)\geq|W_j|\cdot(n+1)/2-(\varepsilon^\prime n)^2&=|W_j|\cdot(1/2+\varepsilon^\prime)n-\varepsilon^\prime n|W_j|-(\varepsilon^\prime n)^2\nonumber\\
&\geq|W_j|\cdot(|W_1|+|W_2|)-\varepsilon^\prime n^2/2
\end{align}

\smallskip

\noindent Secondly, since for each $j\in\{3, 4\}$, $e^+(W_1\cup W_4, W_j)\leq e^+(U_1, U_2)\leq(\varepsilon^\prime n)^2$, and by Claim \ref{ccclaim}, we get that
\begin{align}\label{W5}
e^+(W_2\cup W_3, W_j)\geq(n+1)/2\cdot|W_j|-(\varepsilon^\prime n)^2&=(1/2+\varepsilon^\prime)n\cdot|W_j|-\varepsilon^\prime n|W_3|-(\varepsilon^\prime n)^2\nonumber\\
&\geq(|W_2|+|W_3|)\cdot|W_j|-\varepsilon^\prime n^2/2.
\end{align}
Inequalities $(\ref{W3})$ and $(\ref{W5})$ imply that the property $(3.4)$ of \textbf{EC3} holds.

\smallskip

Further, for each $i\in\{1, 3\}$, we define $E_i$ to be the set of vertices $u$ in $W_i$ with $d^\sigma_{W_i}(u)\leq\varepsilon^{1/3}|W_i|$ for some $\sigma\in\{+, -\}$. Clearly, $|E_i|\leq\sqrt{10\varepsilon}|W_i|$. Further, if there exists a vertex $x$ in $E_1\cup W_2\cup W_4$ (resp., a vertex $y$ in $(E_3\cup W_2\cup W_4$) such that for each $\sigma\in\{+, -\}$, $d_{W_3}^\sigma(x)>\varepsilon^{1/3}|W_3|)$ (resp., $d_{W_1}^\sigma(y)>\varepsilon^{1/3}|W_1|)$, then we put $x$ (resp., $y$) into the vertex set $W_3$ (resp., $W_1$) and update the vertex sets $W_1$ and $W_3$. Repeat the above operation until no such vertices $x$ and $y$ exist. Then let $W_i:=W_i\setminus E_i$ for each $i\in\{1, 3\}$.

\smallskip

Symmetrically, recall that $D[W_2\cup W_4]$ is an $\varepsilon$-almost complete bipartite digraph. For any $j\in\{2, 4\}$, let $E_i$ be the set of those vertices $u$ in $W_j$ such that $b_{W_{j+2}}(u)\leq\varepsilon^{1/3}|W_{j+2}|$, where the subscript of $W_{j+2}$ takes the remainder of modulo for $j=4$. Clearly, $|E_{j}\cap W_j|\leq\sqrt{10\varepsilon}|W_j|$ for every $j\in\{2, 4\}$. If there exists a vertex $x$ in $E_{4}\cup (V(D)\setminus(W_1\cup W_3))$ (resp., a vertex $y$ in $E_{4}\cup (V(D)\setminus(W_1\cup W_3))$) such that $s_{W_4}(x)>\varepsilon^{1/3}|W_4|$ (resp., $s_{W_2}(y)>\varepsilon^{1/3}|W_2|$), then we put $x$ (resp., $y$) into the vertex set $W_4$ (resp., $W_2$) and update the sets $W_2$ and $W_4$. We repeat this operation until on such vertices $x$ and $y$ exist. Then let $W_2:==W_2\setminus E_2$ and $W_4:=W_4\setminus E_4$. 
Finally set $W_5=\bigcup_{i=1}^4 E_i$, and then together with $\delta^0(D)\geq(n+1)/2$, the properties $(3.1)$, $(3.2)$ and $(3.3)$ hold.

\smallskip

Hence the proof of Lemma \ref{claim1} is completed. \hfill $\Box$

\end{document}